\documentclass[a4paper,11pt]{article}
\usepackage[cp1250]{inputenc}
\usepackage[T1]{fontenc}
\usepackage{amsfonts,amsmath,amssymb,amsthm,bbm,enumerate,enumitem,geometry,wrapfig,url,yfonts,fancyhdr}
\usepackage{subfigure,graphicx,color,tasks}
\usepackage{pdfpages}
\everymath{\displaystyle}

\theoremstyle{plain} 

\newtheorem{tw}{Theorem}[section]	
\newtheorem{corollary}[tw]{Corollary}

\newtheorem{pro}[tw]{Proposition}
\newtheorem{lemma}[tw]{Lemma}
\theoremstyle{definition} 
\newtheorem{definition}[tw]{Definition}
\newtheorem{example}[tw]{Example}

\newcommand\mN{{\mathbb N}}

\newcommand\bx{\mathbf{x}}
\newcommand\by{\mathbf{y}}
\newcommand\cA{{\mathcal A}} 

\newcommand\cF{{\mathcal F}}

\renewcommand{\c}{\circ}
\newcommand\st{\star}

\newcommand\cMm{\mathcal{M}_{(X,\cA)}}
\newcommand\cFf{\cF_{(X,Y)}} 

\newcommand{\sA}{\mathsf{A}}
\newcommand{\rH}{\mathsf{H}} 
\newcommand{\riH}{\mathsf{iH}}

\newcommand\trd{\,\triangledown\,}
\newcommand\trdd{\triangledown}

\newcommand{\rJ}{\mathrm{J}}

\newcommand\rI[1]{\mathrm{Su}(#1)}
\newcommand\rIm[2][\mu]{\mathrm{Su}(#1,#2)}

\newcommand\cg[2][\c]{\mathrm{Su}^{#1}_{2}(#2)}   
\newcommand\cgn[3][\c]{\mathrm{Su}^{#1}_{#3}(#2)} 

\newcommand\cgnm[3][\mu]{\mathrm{Su}^{\c}_{#3}(#1,#2)} 

\newcommand\cgd[2][\st]{\mathrm{Su}_{#1}^{2}(#2)}
\newcommand\cgdn[3][\st]{\mathrm{Su}_{#1}^{#3}(#2)}

\newcommand\cgdnm[3][\mu]{\mathrm{Su}^{#3}_{\st}(#1,#2)}

\newcommand\bmm[5][\mu]{\mathrm{I}_{#2}^{#3,#4}(#1,#5)} 

\renewcommand\ge{\geqslant}
\renewcommand\le{\leqslant}

\newcommand{\mI}[1]{\mathbbm{1}_{#1}}

\usepackage{xcolor}
\definecolor{darkgreen}{rgb}{0,0.5,0}

\title{New monotone measure-based integrals \\ inspired by scientific impact problem} 
\author{Micha{\l} Boczek$^{\textrm{a},}$\footnote{Corresponding author}, Anton Hovana$^{\textrm{b}},$ Ondrej Hutn\'ik$^{\textrm{b}},$ Marek Kaluszka$^{\textrm{a},}$\footnote{
\textit{E-mail adressess}: michal.boczek.1@p.lodz.pl, ondrej.hutnik@upjs.sk, anton.hovana@student.upjs.sk, kaluszka@p.lodz.pl
}
\\ 
{\emph{
\small{$^{\textrm{a}}$Institute of Mathematics, Lodz University of Technology, 90-924 Lodz, Poland}}}\\
\emph{ \small{$^{\textrm{b}}$Institute of Mathematics, Pavol Jozef \v Saf\'arik University in Ko\v sice, 040-01 Ko\v sice,
Slovakia}}
}
\date{}

\begin{document}
\maketitle

\begin{abstract}
In this paper, we define new functionals generalizing scientometric indices proposed by Mesiar and G\k{a}golewski in 2016 to overcome some limitations of $h$-index. These functionals are integrals with respect to a~monotone measure as well as aggregation functions under some mild conditions. We derive  numerous properties of the new integrals and analyze subadditivity property in detail. We also give a~partial solution to the problem posed by Mesiar and  Stup\v nanov\'a to find an algorithm for computing the pseudo-decomposition integral of $n$-th order based on operations $\oplus=+$ and $\odot=\wedge,$ which will be useful in multi-criteria decision problems. 
\end{abstract}


\section{Introduction}
In order to compare the efficiency of work of two researchers, one must construct a~rule  that is the most objective and fair. It turns out that the task is very difficult. Currently, there are many scientometric indices known in the literature. Their calculations are based on two inputs: number of publications and  number of citations of each publication (measuring the quality and importance of publications). Nowadays, the most popular scientometric index is $h$-index introduced in 2005 by Hirsch \cite{hirsch}. It is implemented in the largest scientific databases such as Scopus, or WoS. An axiomatic approach explaining the nature of $h$-index can be found in papers 
\cite{brandao,Miroiu,Quesada2,woeginger},
whereas its mathematical properties can be found in \cite{Franze}. 
Torra and Narukawa \cite{torra2} proved that $h$-index is the Sugeno integral with respect to the counting measure. And because of specificity of the Sugeno integral, $h$-index has some drawbacks, see \cite{mingers}. For example, let's consider two researchers possessing the same number of papers, say $n,$ but each paper of the first one is cited $n$-times, and each paper of the second one is cited exactly $3n$-times. Using the criterion of importance of author's publication and the number of quotations for each paper, one can see that the second researcher should have a~higher scientometric index
if both authors work in the same domain and have similar research experience measured by the years of work. However, $h$-index of both authors is the same and is equal to $n.$ To overcome the above limitations, Mesiar and G\k{a}golewski \cite{mesiar14} have proposed two new indices based on the idea of $h$-index (for more details, see Section \ref{sec:application}).

The present paper introduces and studies properties of two functionals: upper $n$-Sugeno integral and lower $n$-Sugeno integral of a~measurable function with respect to a~monotone measure. Both functionals are 
integrals in the spirit of definition introduced here. 
Integrals have many applications, especially in multicriteria decision theory, economy, optimalization or data classification \cite{bo,Ch,coucerio,dubois,greco4,hirsch,torra06,torra2,zhai}. To the best of our knowledge, there are only very few papers which describe the connection between integrals and scientometric indices, see~\cite{gag,torra2}. From those papers it follows that total number of citations is related to Choquet integral, $h$-index and Kosmulski $h(2)$-index are related to Sugeno integral, Kosmulski MAXPROD index 
is related to Shilkret integral, etc. In this paper we further show that the upper and lower $n$-Sugeno integrals generalize scientometric indices introduced by Mesiar \& G\k{a}golewski and others, e.g. generalized Kosmulski index \cite{deineko}, iterated $h$-index \cite{garcia}, 
$\rH_\alpha$ and $\rH^\beta$ indices \cite{JMS}.

The second main result of the work is a~relation between the lower $n$-Sugeno integral and some special pseudo-decomposition integral introduced by Mesiar and Stup\v{n}anov\'{a} in \cite{mesiar13a} (see Theorem~\ref{tw4.10} below). Their question from \cite[Conclusion]{mesiar13a} motivated us to describe an algorithm for computation of the pseudo-decomposition integral. 

Our paper is organized as follows. In Section \ref{sec:not} we provide basic notations and definitions we work with. In Sections \ref{sec:upper} and \ref{sec:lower} we introduce new concepts of upper and lower $n$-Sugeno  integral, examine their basic properties and provide their equivalent forms.  Section~\ref{sec:application} includes applications of the obtained results mainly to aggregation and scientometrics. For better readability we postpone some technical proofs of our statements to Appendix.

\section{Basic notations and preliminaries}\label{sec:not}

Let $(X,\cA)$ be a~measurable space, where $\cA$ is a~$\sigma$-algebra of subsets of a~non-empty set $X.$
The class of all measurable functions $f\colon X\to Y,$ where $Y=[0,\bar{y}]$ for $0<\bar{y}\le \infty,$ is denoted by $\cFf.$ Usually, we  take $\bar{y}=1$ or $\bar{y}=\infty.$
A~{\it monotone measure}  on $\cA$  is a~nondecreasing set function $\mu\colon \cA\to [0,\infty],$ i.e., $\mu(A)\le\mu(B)$ whenever $A\subset B$ with $\mu(\emptyset)=0$ and $\mu(X)>0.$ The range of $\mu$ we write as $\mu(\cA).$ We denote by $\cMm$ the class of all monotone measures  on $(X,\cA).$ Given $f,g\in\cFf$ and $\mu\in\cMm,$ we say that \textit{$g$ dominates $f$} with respect to  $\mu$ and write $f\le_\mu g$ if $\mu(\{f\ge t\})\le \mu(\{g\ge t\})$ for all $t,$ where $\{f\ge t\}=\{x\in X\colon f(x)\ge t\}.$    Hereafter, $a\wedge b=\min(a,b)$ and $a\vee b=\max(a,b).$  
We say that a~function $\c\colon Y_1\times Y_2\to [0,\infty]$ is \textit{nondecreasing} if $a_1\c a_2\le b_1\c b_2$ whenever $a_i\le b_i,$ where $a_i,b_i\in Y_i\subset [0,\infty]$ for $i=1,2.$

Sugeno integral of $f\in\cFf$ with respect to $\mu\in\cMm$ \cite{sug,wang} is  defined by
\begin{align}\label{sugeno}
\rIm{f}:=\sup_{t\in Y}\{t\wedge \mu (\{ f\ge t\})\}.
\end{align}
To this day, many researchers introduced numerous generalizations of the Sugeno integral like  generalized upper Sugeno integral, pseudo-decomposition integral or q-integral for $\bar{y}=\mu(X)=1,$
and studied   their properties \cite{boczek9,hutnik1,dubois,kaluszka,mesiar13a,suarez}.

To make our paper as self-contained as it gets, we give some properties of the Sugeno integral that we follow later. Hereafter, 
$c\searrow a$ and $c\nearrow a$ means that $c\to a$ for  $c>a$ and  $c<a,$ respectively.

\begin{lemma}\label{lemat}
Let $(\mu,f)\in \cMm\times \cFf.$ The Sugeno integral possesses the following properties:
\begin{enumerate}[noitemsep]
\item[(a)] $\rIm{f}\in Y,$
\item[(b)]  $t>\mu(\{f\ge t\})$ for $t>\rIm{f}$ and $t<\mu(\{f\ge t\})$ for $t<\rIm{f},$
\item[(c)] $\rIm{f}=\lim _{t\nearrow \rIm{f}}(t\wedge \mu (\{f\ge t\}))$ if $\rIm{f}>0,$ 
\item[(d)] $\rIm{f}=\lim _{t\searrow \rIm{f}}(t\vee \mu (\{f>t\}))$ if $\rIm{f}<\bar{y},$ 
\item[(e)] $\rIm{f}=0$ if and only if $\mu(\{ f\ge t\} )=0$ for all $t>0.$
\end{enumerate}
\end{lemma}
\begin{proof} 
Properties (a) and (b) follow from 
\eqref{sugeno}, since $Y=[0,\bar{y}]$ (see also \cite[Lemma~9.7]{wang}). Properties (c) and (d) follow from (b) as 
\begin{align*}
\rIm{f}&=\lim _{t\nearrow \rIm{f}}t=\lim _{t\nearrow \rIm{f}}\big(t\wedge \mu (\{f\ge t\})\big),\\
\rIm{f}&=\lim _{t\searrow \rIm{f}}t= \lim _{t\searrow \rIm{f}}\big(t\vee \mu (\{f> t\})\big).
\end{align*}
To prove (e), by the definition of the Sugeno integral we have that
\begin{align*}
0=\sup_{t\ge 0}\{ t\wedge \mu (\{ f\ge t\} )\}=(0\wedge \mu(X))\vee \sup_{t>0}\{ t\wedge \mu(\{ f\ge t\} )\}.
\end{align*}
Now it is evident that $\mu(\{ f\ge t\})=0$ for all $t>0.$
\end{proof}

We formulate 
 properties
which any integral should possess.

\begin{definition}\label{defcalka}
A~functional  $\rJ\colon \cMm\times \cFf\to [0,\infty]$
is called an {\it integral}
if 
\begin{enumerate}[noitemsep]
\item[$(C_1)$] 
$\rJ(\mu,f)\le  \rJ(\mu,g)$ 
whenever $f\le_\mu g,$
\item[$(C_2)$]  
$\rJ(\mu,f)\le\rJ(\nu,f)$ 
whenever $\mu(A)\le\nu(A)$ for all $A\in \cA,$ 
\item[$(C_3)$] 
$\rJ(\mu,a\mI{A})=s(a,\mu(A))$ 
for all $a\in Y$  and $A\in\cA,$ where 
$s\colon Y\times [0,\infty] \to Y$ is a~nondecreasing function 
such that $s(a,0)=s(0,b)=0$ for any $a,b.$ 
\end{enumerate}
\end{definition}

 For a~fixed $\mu\in\cMm,$ the property $(C_3)$ is known in the literature as $\mu$-generated property of the integral 
$\rJ,$ see \cite[Definition 3.3]{kawabe1}. 
Restriction to the class of monotone measures with $\mu(X)=1$ and $Y=[0,1]$ in Definition~\ref{defcalka} is closely related to fuzzy integral introduced by Struk \cite{struk}. In fact, the condition $(C_3)$ with $s(a,1)=a=s(1,a)$ for any $a\in [0,1]$ implies the conditions (2) and (3) in \cite[Definition 1]{struk}. However, the assumption (1) from \cite[Definition 1]{struk} is stronger than $(C_1).$ 
Examples of integrals of nonnegative functions with respect to monotone measures include the Choquet integral 
\cite{choquet}, Sugeno integral or generalized upper Sugeno integral (see formula (2) in~\cite{bo}) under some additional restrictions.

\section{Upper $n$-Sugeno integral}\label{sec:upper}
In this section, we introduce a~new type of integral with respect to a~monotone measure. Our motivation for doing so comes from the 
lower $2$-$h$-index defined by Mesiar and G\k{a}golewski (see~\eqref{hd}).

We say that  $\c\colon Y\times Y \to Y$ is an {\it admissible fusion map} if it is nondecreasing and $0 \c a\le a$ for all  $a\in Y.$  The most important examples for $Y=[0,\infty]$ are:  the standard addition, pseudo-addition \cite{benvenuti}, the standard product,
minimum, maximum or means \cite{beliakov}.
Moreover, for $Y=[0,1]$ 
 the examples are: boolean conjunctions such as
semicopulas \cite{bas,bo,dur}, copulas \cite{dur2}, $t$-norms, conjunctive aggregations \cite{beliakov} and  fuzzy conjunctions \cite{dubois}, and other binary operations like  uninorms, $t$-semiconorms or averaging aggregations \cite{beliakov}.

\begin{definition}\label{def_upper_sug}
Let  $(\mu,f)\in \cMm\times \cFf$ and $\c$ be an admissible fusion map. For $n\ge 1$ the \textit{upper $n$-Sugeno integral} is defined using the recurrence
\begin{align*}
\cgnm{f}{n+1}:=\sup_{t\in Y}\big\{ (t\c \cgnm{f}{n})\wedge \mu(\{ f\ge t\})\big\}
\end{align*}
with the initial condition $\cgnm{f}{1}:=\rIm{f}.$
\end{definition}

Lemma~\ref{lemat}\,(a) yields  $\cgnm{f}{1}\in Y.$  The induction 
implies  that $\cgnm{f}{n}\in Y$ for all $n\ge 2,$ so the functional $\cgnm{f}{n}$ in Definition~\ref{def_upper_sug} is well-defined. We show that
the upper $n$-Sugeno integral
is an integral in the sense of Definition \ref{defcalka}.

\begin{pro}\label{pro3.2}
Let $n\ge 1,$ $f,g\in\cFf$ and $\mu,\nu\in\cMm.$  Then
\begin{enumerate}[noitemsep]
\item[(a)] $\cgnm{f}{n}\le \cgnm{g}{n}$ whenever $f\le_\mu g.$ 
\item[(b)] $\cgnm[\mu]{f}{n}\le \cgnm[\nu]{f}{n}$ whenever $\mu(A)\le\nu(A)$ for all $A\in \cA.$ 
\item[(c)] $\cgnm{a\mI{A}}{n}=s_n(a,\mu(A))$ for $a\in Y$ and $A\in\cA,$ where $s_n\colon Y\times [0,\infty]\to Y$ is a~nondecreasing function 
such that  $s_n(a,0)=s_n(0,b)=0$ and $s_{n+1}(a,b)=(a\c s_n(a,b))\wedge b$ for all $a,b.$
\end{enumerate}
\end{pro}
\begin{proof} 
Properties $(a)$ and $(b)$ follow immediately from 
Definition~\ref{def_upper_sug} and monotonicity of $\c.$
We shall prove by induction that $(c)$ holds. 
Clearly, $\rIm{a\mI{A}}=a\wedge b$ with $b=\mu(A).$
If (c) is true for some $n\ge 1,$ by monotonicity of function
$t\mapsto t\c b$ we have 
\begin{align*}
\cgnm{a\mI{A}}{n+1}&=\big[(0\c \cgnm{a\mI{A}}{n})\wedge \mu(X)\big] \vee \sup_{t\in (0,\bar{y}]}\big\{ (t\c \cgnm{a\mI{A}}{n})\wedge \mu(\{a\mI{A}\ge t\})\big\}\notag
\\&=(0\c s_{n}(a,b))\vee\big[(a\c s_n(a,b))\wedge b\big]\notag
\\&=(a\c s_n(a,b))\wedge b=s_{n+1}(a,b),
\end{align*}  
as $0\c s_n(a,b)\le s_n(a,b)= \cgnm{a\mI{A}}{n}\le b \le \mu(X)$ and $0\c s_n(a,b)\le a\c s_n(a,b).$
The induction hypothesis implies that $s_{n+1}$ is nondecreasing 
and  $s_{n+1}(a,0)=s_{n+1}(0,b)=0,$ as desired.  
\end{proof}

Now, we provide other properties of the upper $n$-Sugeno integral.  From now on, to shorten the notation, we write $\cgn{f}{n}$ and $\rI{f}$ instead of $\cgnm{f}{n}$ and  $\rIm{f},$ respectively f there is no ambiguity.

\begin{pro}\label{pro3.3}
\begin{enumerate}[noitemsep]
\item[(a)]    
Let $(\mu,f)\in\cMm\times \cFf.$
If $\rI{f}=0,$ then $\cgn{f}{n}=0$ for all $n.$ Moreover, if $\cgn{f}{k}=0$ for some  $k>1$ and 
$a\c b>0$ for all $a,b>0,$ then $\cgn{f}{n}=0$ for  any $n.$
\item[(b)] 
$\cgn{af}{n}\le a\,\cgn{f}{n}$ for some $a>1$ and 
for all $(\mu,af)\in\cMm\times \cFf$ and $n\ge 1$
provided that $(ax)\c (ay)\le a(x\c y)$ for all  $ax,ay\in Y.$
Moreover, $\cgn{af}{n}\ge a\,\cgn{f}{n}$ for some $a\in (0,1)$ and 
all $(\mu,f)\in\cMm\times \cFf$ whenever 
$(ax)\c (ay)\ge a(x\c y)$ for all $x,y\in Y.$ 
\item[(c)] $\mu(A)=\cgn{\bar{y}\mI{A}}{n}$ for any $A\in \cA$ and $n\ge 1$ whenever $\mu(X)\le \bar{y}$ and
$\bar{y}\c b\ge b$ for all  $b\in Y.$
\item[(d)] (Idempotency) 
$\cgn{a\mI{X}}{n}=a$ for all $a\in Y$ and $n\ge 1$ if and only if $\mu(X)\ge \bar{y}$ and $a\c a=a$ for any $a\in Y.$ 
\end{enumerate}
\end{pro}
\begin{proof}
$(a)$ If $\rI{f}=0,$ then 
by Lemma~\ref{lemat}\,(e) we have $\mu(\{ f\ge t\})=0$ for all $t>0.$ Hence
\begin{align*}
\cg{f}=\big[(0\c 0)\wedge \mu(X)\big]\vee \sup_{t\in (0,\bar{y}]}\{ (t\c 0)\wedge 0\}=0,
\end{align*}
as $0\c 0=0.$ Applying the induction, we will prove that $\cgn{f}{n}=0$ for all $n.$

Assume that $\cgn{f}{k}=0$ for some $k>1.$ Thus $t\c \cgn{f}{k-1}=0$ for all $t>0$ or $\mu(\{ f\ge t\})=0$ for all $t>0.$
Suppose that $\cgn{f}{k-1}>0.$ Then by the assumption on $\c,$ we have  $\mu(\{ f\ge t\})=0$ for all $t>0,$ so 
$\rI{f}=0,$ which implies that $\cgn{f}{k-1}=0,$ a~contradiction. Therefore $\cgn{f}{k-1}=0,$ which leads to $\rI{f}=0,$ and 
so $\cgn{f}{n}=0$ for all $n.$
 
\noindent The proof of part $(b)$ is again by induction on $n.$ Clearly, 
\begin{align*}
\rI{af}&=\sup_{t\in Y}
\big\{(a(t/a))\wedge \mu(\{f\ge t/a\})\big\}
\\&\le \sup_{s\le \bar{y}/a}\big\{(as)\wedge (a\mu(\{f\ge s\}))\big\}=a\,\rI{f}
\end{align*} 
for $a>1.$ 
Assume that the assertion holds for some $n\ge 1.$ Then, by induction hypothesis, 
\begin{align*} 
\cgn{af}{n+1}&=\sup_{s\le \bar{y}/a}\big\{\big( (as)\c \cgn{af}{n}\big)\wedge \mu(\{ f\ge s\})\big\}\\
&\le \sup_{s\le \bar{y}/a}\big\{\big( (as)\c (a\cgn{f}{n})\big)\wedge \big(a\mu(\{ f\ge s\})\big)\big\}\\
&\le a\sup_{s\le \bar{y}/a}\big\{(s\c \cgn{f}{n} )\wedge \mu(\{ f\ge s\})\big\}= a\,\cgn{f}{n+1}.
\end{align*}
The proof for the case $0<a<1$ is analogous. To prove (c) and (d) one can use Proposition~\ref{pro3.2}\,(c).
\end{proof} 

\begin{pro}\label{pro3.4}
The sequence $(\cgn{f}{n})_{n\ge 1}$  is nondecreasing for all 
$(\mu,f)\in \cMm\times\cFf$
 if and only if   
 $a\c b\ge a\wedge b$ for all $a,b\in Y.$
\end{pro}
\begin{proof}
 ``$\Rightarrow$'' 
By Proposition~\ref{pro3.2}\,(c) we get $\cg{a\mI{A}}=
(a\c(a\wedge b))\wedge b,$ where $A\in \cA,$ $a\in Y$ and  $b=\mu(A).$ Since $\cgn{a\mI{A}}{2}\ge  \cgn{a\mI{A}}{1},$ we have $
(a\c (a\wedge b))\wedge b\ge a\wedge b$
for all $a\in Y,$ $b\in \mu(\cA)\cap Y$ and any monotone measure $\mu.$ 
Hence, 
$a\c b\ge (a\c b)\wedge b \ge (a\c(a\wedge b))\wedge b\ge a\wedge b$
for all $a,b\in Y.$ 

\noindent ``$\Leftarrow$''  By assumption
$\cg{f}\ge \sup _{t\in Y}\big\{(t\wedge \rI{f})\wedge \mu(\{f\ge t\})\big\}=\rI{f}.$
Suppose  that $\cgn{f}{n}\ge\cgn{f}{n-1}$ for some $n>1.$ By the monotonicity $t\mapsto a\circ t$ for all $a$ and the induction hypothesis, we obtain
\begin{align*}
\cgn{f}{n+1}&=\sup_{t\in Y}\big\{(t\c \cgn{f}{n})\wedge \mu(\{ f\ge t\})\big\}\\&\ge \sup_{t\in Y}\big\{(t\c \cgn{f}{n-1})\wedge \mu(\{ f\ge t\})\big\}
=\cgn{f}{n},
\end{align*}
thus $(\cgn{f}{n})_{n\ge1}$ is a~nondecreasing sequence.
\end{proof}

Let $(\mu,f)\in \cMm\times \cFf.$  Recall that  
\begin{align}\label{sug_rown1}
\rI{f}=\inf_{t\in Y} \{t\vee \mu (\{f> t\})\}=\sup_{A\in\cA}\{\inf_{x\in A} f(x) \wedge \mu(A)\},
\end{align}
see \cite{bo,sug,wang}. We present formulas for the  upper $n$-Sugeno integral, which have the forms as in \eqref{sug_rown1}.

\begin{tw}\label{tw3.5} Let $\c$ be an~admissible fusion map that is continuous in the first argument.  
Then for all $(\mu,f)\in \cMm \times \cFf$ and $n\ge 2$ we have
\begin{align*}
\cgn{f}{n}=\inf_{t\in Y} \big\{\big(t\c\cgn{f}{n-1}\big)\vee \mu (\{f> t\})\big\}.
\end{align*}
\end{tw}
\begin{proof}
\noindent The proof is given in Appendix.
\end{proof}

\begin{tw}\label{tw3.6}
For all $(\mu,f)\in \cMm \times \cFf$ and $n\ge 2$ we have
\begin{align*}
\cgn{f}{n}=\sup_{A\in\cA}\big\{\big(\inf_{x\in A} f(x)\c\cgn{f}{n-1}\big) \wedge \mu(A)\big\}.
\end{align*}
\end{tw}
\begin{proof} Arguing  as in the proof of  Theorem $2.2$ in \cite{hutnik1},
we shall show  more than it is needed, i.e.,
if a~map  $\c$ is {\it nondecreasing in the first coordinate}, then 
\begin{align}\label{m1}
\sup_{t\in Y} \{(t\c a)\wedge \mu(\{f\ge t\})\}=\sup_{A\in\cA}\{(\inf_{x\in A} f(x)\c a) \wedge \mu(A)\}
\end{align}
for all $a\in Y.$ In fact, 
let $t\in Y$ and $A_t=\{f\ge t\}.$ Thus, $\inf_{x\in A_t} f(x)\ge t$ and
\begin{align*}
(t\c a)\wedge \mu(\{f\ge t\})\le (\inf_{x\in A_t} f(x)\c a)\wedge \mu(A_t)\le \sup_{A\in\cA}\{(\inf_{x\in A} f(x)\c a) \wedge \mu(A)\}.
\end{align*}
Therefore, the left hand side in \eqref{m1} is not greater than the right one. Let $A\in\cA$ and $t_0=\inf_{x\in A} f(x).$ Then 
$A\subset \{f\ge t_0\}$ and
\begin{align*}
(\inf_{x\in A} f(x)\c a)\wedge\mu(A)\le (t_0\c a)\wedge\mu(\{f\ge t_0\})\le \sup_{t\in Y}\{(t\c a)\wedge\mu(\{f\ge t\})\}.
\end{align*}
Thus, the left hand side in \eqref{m1} is  greater than or equal to the right one, which finishes the proof.
\end{proof}

The most important property of an integral is subadditivity. The following  concept was introduced in \cite{bo}.

\begin{definition}\label{defsub}
Let $\mu\in \cMm.$ We say that $f,g\in\cFf$ are $\mu$-{\it subadditive} with respect to a~map  $\trdd\colon \mu (\cA)\times \mu(\cA)\to \mu (\cA)$ (\textit{$\mu$-$\trdd$-subadditive} for short) if 
\begin{align*}
\mu(\{f> a\}\cup \{g> b\})\le \mu(\{f> a\})\trd \mu(\{g> b\})
\end{align*}for all $a,b\in Y.$
\end{definition}
Note that all functions $f,g$ are $\mu$-$\trdd$-subadditive if  $a\trd b=(a+b)\wedge \mu(X)$ whenever the~monotone measure $\mu$ is {\it subadditive}, that is, $\mu(A\cup B)\le\mu(A)+\mu(B)$ for all $A,B\in\cA.$ Recall that  $f,g\in\cFf$ are {\it comonotone} 
if $(f(x)-f(y))(g(x)-g(y))\ge 0$ for all $x,y\in X.$ 
Equivalently, $f$ and $g$ are comonotone, if for any  $t\in Y$ either $\{ f> t\} \subset  \{ g> t\}$ or $\{g> t\}\subset \{ f> t\}.$ Thus, comonotone functions are $\mu$-$\trdd$-subadditive 
for
$\trdd\ge \vee.$   
For more examples of $\mu$-$\trdd$-subadditive functions we refer to \cite{bo}.

\begin{tw}\label{tw3.8}
Assume that  $\trdd\colon \mu(\cA)\times \mu(\cA)\to \mu(\cA),$ $f,g\in\cFf$ and
\begin{align}\label{M}
[((a+ b)\wedge \bar{y})\odot ((c+d)\wedge \bar{y})]\vee (\alpha \trd \beta)&\le [(a\odot c)\vee \alpha]+ [(b\odot d)\vee \beta]
\end{align} 
for  $a,b,c,d\in Y$  and $\alpha,\beta\in \mu(\cA)$ such that $a+b,c+d\in Y$ with
$\odot\in \{\c,\mathrm{P}\},$
where the admissible fusion map $\c$ is continuous in the first argument and $x \mathrm{P} y=x$ for any $x,y.$
If $f,g$ are $\mu$-$\trdd$-subadditive and  $f+g\in \cFf,$ then 
\begin{align}\label{Minkowski}
\cgn{f+g}{n}\le\cgn{f}{n}+ \cgn{g}{n}
\end{align}
for $n\ge 1.$
Moreover, $\mu$ is a~subadditive monotone measure whenever there is some $n\ge 1$ such that  \eqref{Minkowski} holds for all $f,g\in \cFf$ with $\mu(X)\le \bar{y}$ and  $\bar{y}\c b\ge b$ for all $b\in Y.$
\end{tw}
\begin{proof} 
The proof is by induction on $n.$ Let $n=1.$   
Evidently, $\{f\le a\}\cap \{g\le b\}\subset \{f+g\le a+b\}.$ As $f(x)+g(x)\le \bar{y}$ for any $x$ and $f,g$ are $\mu$-$\trdd$-subadditive, we have
\begin{align*}
[(a+ b)\wedge \bar{y}]\vee \mu( \{f+g>(a+b)\wedge \bar{y}\})&\le [(a+ b)\wedge \bar{y}]\vee [\mu( \{f>a\})\trd \mu(\{g>b\})]
\\&\le [a\vee  \mu(\{ f>a\}) ]+[b\vee \mu(\{ g>b\})],
\end{align*}
where the last inequality holds by \eqref{M} with $x\odot y=x.$ By \eqref{sug_rown1} we get
\begin{align*}
\rI{f+g}\le [a\vee  \mu(\{ f>a\}) ]+[b\vee \mu(\{ g>b\})].
\end{align*}
Taking the lower bound for $a,b\in Y$ finishes the proof of \eqref{Minkowski} for $n=1.$

Assume that \eqref{Minkowski} holds for some $n>1.$
By $\mu$-$\trdd$-subadditivity and \eqref{M} with  $\odot=\c,$
we get 
\begin{align*}
\Big[\big[(a+ b)\wedge \bar{y}\big]\c&\big[(\cgn{f}{n}+ \cgn{g}{n})\wedge \bar{y}\big]\Big]\vee \mu( \{f+g>(a+b)\wedge \bar{y}\})\nonumber
\\&\le \Big[\big[(a+ b)\wedge \bar{y}\big]\c\big[(\cgn{f}{n}+ \cgn{g}{n})\wedge \bar{y}\big]\Big]\vee [\mu( \{f>a\})\trd \mu(\{g>b\})]
\\&\le \big[(a\c\cgn{f}{n})\vee  \mu(\{ f>a\}) \big]+\big[(b\c\cgn{g}{n})\vee \mu(\{ g>b\}) \big].
\end{align*}
By the induction hypothesis and $\cgn{f+g}{n}\le \bar{y},$
we have for all $a,b\in Y$
\begin{align*}
\big[[(a+ b)\wedge \bar{y}]\c\cgn{f+g}{n}\big]&\vee \mu( \{f+g>(a+b)\wedge \bar{y}\})\nonumber
\\&\le \big[(a\c\cgn{f}{n})\vee  \mu(\{ f>a\}) \big]+\big[(b\c\cgn{g}{n})\vee \mu(\{ g>b\}) \big].
\end{align*}
As a~consequence of Theorem~\ref{tw3.5} we obtain
\begin{align*}
\cgn{f+g}{n+1}\le \big[(a\c\cgn{f}{n})\vee  \mu(\{ f>a\}) \big]+\big[(b\c\cgn{g}{n})\vee \mu(\{ g>b\}) \big].
\end{align*}
Taking the lower  bound for $a$ and then for $b$ we finish the proof of \eqref{Minkowski}.

Suppose that~\eqref{Minkowski} is satisfied for 
some $n\ge 1.$
Then Proposition~\ref{pro3.3}\,(c) yields 
$\cgn{\bar{y}\mI{D}}{k}=\mu(D)$ for all $D$ and $k.$
Putting $f=\bar{y}\mI{A}$ and $g=\bar{y}\mI{B\backslash A}$ in~\eqref{Minkowski}, 
we obtain $
\mu(A\cup B)\le \mu(A)+\mu(B\backslash A)\le \mu(A)+\mu(B).$ This completes the proof.
\end{proof}

\begin{corollary}\label{cor1} 
If $\mu\in\cMm$ is subadditive
and $\mu(X)\le \bar{y},$ 
 then
\begin{align}\label{sub}
\cgn[+]{f+g}{n} \le \cgn[+]{f}{n}+\cgn[+]{g}{n}
\end{align}
for any $n\ge 1$ and all $f,g\in\cFf$ such that  $f+g\in \cFf.$
Moreover, if there is $n$ such that  \eqref{sub} holds for all $f,g\in \cFf,$ then 
$\mu$ is a~subadditive monotone measure. 
\end{corollary} 
\begin{proof}
The assertion follows from Theorem~\ref{tw3.8} for 
$a\c b=(a+b)\wedge \bar{y}$ and $a\trd b=(a+b)\wedge \mu(X).$ 
\end{proof}

\begin{corollary}\label{cor2} 
If $f,g\in\cFf$ are comonotone functions
and $\mu(\cA)\subset Y,$  then
$\cgn[+]{f+g}{n} \le \cgn[+]{f}{n}+\cgn[+]{g}{n}$
for all $n\ge 1$ and $f+g\in \cFf.$ 
\end{corollary} 
\begin{proof}
Apply Theorem~\ref{tw3.8} for $a\c b=(a+b)\wedge\bar{y}$ and $\trdd=\vee.$
\end{proof}

Condition \eqref{M} is  valid if $\trdd\le +$ and $[(a+b)\wedge\bar{y}]\c[(c+d)\wedge\bar{y}]\le (a\c c)+(b\c d)$ for any  $a,b,c,d\in Y.$ Examples of maps $\c$ satisfying the last inequality are:
    \begin{enumerate}[noitemsep]
    \item[(i)] $a\c b=\lambda a$ for $Y=[0,1]$ and $\lambda \in (0,1],$ or $Y=[0,\infty]$ and $\lambda>0,$
    \item[(ii)] $a\c b=a^\gamma$ with $\gamma\in (0,1),$
    \item[(iii)] $a\c b=\lambda (a+b)+(1-\lambda)(a\vee b)$ for $\lambda\in [0,1]$ and $Y=[0,\infty],$
    \item[(iv)]  $a\c b=a+b-ab$ for  $Y=[0,1].$ 
    \end{enumerate}

\medskip         

As we have shown above, the upper $n$-Sugeno integral possesses several properties of the Sugeno integral,  but not all. Hereafter, $a\mI{X}\vee f = (a\mI{X})\vee f$ and $a\mI{X}\wedge f = (a\mI{X})\wedge f.$ We say that the integral $\rJ$ is \textit{maxitive homogeneous} and \textit{minitive homogeneous} if $\rJ(\mu,a\mI{X}\vee f)=a\vee \rJ(\mu,f)$ and $\rJ(\mu,a\mI{X}\wedge f)=a\wedge \rJ(\mu,f)$ for any $a,$ $\mu,f,$ respectively. Next example demonstrates that it is not the case of upper and lower $n$-Sugeno integral.

\begin{example}\label{ex5}
Let $f=0{.}25\mI{A}+0{.}75\mI{B},$ where $A\cap B=\emptyset$ and $A\cup B=X.$ Assume that $\mu(X)=1$ and $\mu(B)=0{.}5.$ 
It is clear that 
$(1/3)\mI{X}\wedge f=0{.}25\mI{A}+(1/3)\mI{B}$ and $(1/3)\mI{X}\vee f=(1/3)\mI{A}+0{.}75\mI{B}.$ Thus,
$$\cg[+]{f}=0{.}75,\quad  \cg[+]{(1/3)\mI{X}\wedge f}=7/12,\quad \cg[+]{(1/3)\mI{X}\vee f}=5/6,$$
so
$
\cg[+]{(1/3)\mI{X}\wedge f}>(1/3)\wedge\cg[+]{f}$ and
$
\cg[+]{(1/3)\mI{X}\vee f}>(1/3)\vee\cg[+]{f}.$
\end{example}

Now we give one sufficient condition for  minitive/maxitive homogeneity of the integral. 

\begin{pro}\label{pro3.12}
Let $n\ge 2$  and  
$(\mu,f)\in \cMm\times \cFf.$ Then
\begin{enumerate}[noitemsep]
\item[(a)]    
$\cgn{a\mI{X}\wedge f}{n}=a\wedge \cgn{f}{n}$ for any $a\in Y$ if $(a\wedge b)\c (a\wedge c)=a\wedge (b\c c)$ for all  $a,b,c\in Y.$
\item[(b)]    
$\cgn{a\mI{X}\vee f}{n}=a\vee \cgn{f}{n}$  for any $a\in Y$
if $(a\vee b)\c (a\vee c)=a\vee (b\c c)$ for all  $a,b,c\in Y.$
\end{enumerate}
\end{pro}
\begin{proof} 
We show only $(a)$ since the proof of $(b)$ is analogous.  
Since $\rI{a\mI{X}\wedge f}=a\wedge \rI{f},$ we have from Theorem~\ref{tw3.6} that
$
\cgn{a\mI{X}\wedge f}{2}=\sup_{A\in\cA}\big\{\big[(a\wedge \inf_{x\in A} f(x))\c(a\wedge \rI{f})\big] \wedge \mu(A)\big\}=a\wedge \cgn{f}{2}.
$
The second induction step proceeds similarly.
\end{proof}

The admissible fusion map $x\c y=\varphi(x)\wedge \gamma(y)$ satisfies the condition of Proposition \ref{pro3.12}\,(a) if $\varphi,\gamma\colon Y\to Y$ are nondecreasing functions such that $\varphi(a)\wedge \gamma(a)=a$  for any $a\in Y.$
The assumption of Proposition \ref{pro3.12}\,(b) holds if $x\c y=\varphi(x)\vee \gamma(y),$  where the functions $\varphi,\gamma\colon Y\to Y$ are  nondecreasing with $\varphi(a)\vee \gamma(a)=a$  for all $a\in Y.$

\section{Lower $n$-Sugeno integral}\label{sec:lower}

This section is devoted to defining a~new functional generalizing the 
 upper $2$-$h$-index \eqref{hu} and its properties. We say that  $\st\colon [0,\infty]\times Y\to [0,\infty]$  is a~{\it link  map} if it is nondecreasing and
$0\st a\le a$ for all $a\in Y.$ Clearly, the link map coincides with the admissible fusion map if and only if $Y=[0,\infty].$
 
\begin{definition}\label{def_lower_sug}
Let $(\mu,f)\in\cMm\times\cFf$ and $\st$ be a~link map. For $n\ge 1$ the \textit{lower $n$-Sugeno integral} is defined by
\begin{align*}
\cgdnm{f}{n+1}:=\sup_{t\in Y}\big\{t\wedge \big(\mu(\{f\ge t\})\st\cgdnm{f}{n}\big)\big\},
\end{align*}
where $\cgdnm{f}{1}=\rIm{f}.$
\end{definition}

It is clear that $\cgdnm{f}{n}\in Y$ for all $n.$ The next proposition shows that the lower $n$-Sugeno integral satisfies all the properties in Definition~\ref{defcalka}.

\begin{pro}\label{pro4.2}
Let $n\ge 1,$ $f,g\in\cFf$ and $\mu,\nu\in\cMm.$ 
\begin{enumerate}[noitemsep]
\item[(a)] If $f\le_\mu g,$ then $\cgdnm{f}{n}\le\cgdnm{g}{n}.$ 
\item[(b)] If $\mu(A)\le\nu(A)$ for all $A\in \cA,$ then $\cgdnm[\mu]{f}{n}\le\cgdnm[\nu]{f}{n}.$ 
\item[(c)] $\cgdnm{a\mI{A}}{n}=s_n(a,\mu(A))$ 
for $a\in Y$ and $A\in\cA,$
where   
$s_n\colon Y\times [0,\infty]\to Y$ is a~nondecreasing function 
such that  $s_n(a,0)=s_n(0,b)=0$ and $s_{n+1}(a,b)=a\wedge [b\st s_{n}(a,b)]$ for all $a,b.$ 
\end{enumerate}
\end{pro}
\begin{proof} Parts $(a)$ and $(b)$ are immediate by induction.\\
 $(c)$  We have the following recurrence formula 
\begin{align*}
\cgdnm{a\mI{A}}{n+1}&=\big[a\wedge (\mu(A)\st \cgdnm{a\mI{A}}{n})\big]\vee \big[\bar{y}\wedge(0\st \cgdnm{a\mI{A}}{n})\big],
\end{align*}   
where $a\in Y.$ First we show that 
\begin{align}\label{MO5}
\cgdnm{a\mI{A}}{n}\le a
\end{align}
for all  $n$  and $a\in Y.$ We use induction on $n.$
In fact, $\rI{a\mI{A}}=a\wedge \mu(A)\le a.$ Assume that \eqref{MO5} holds for some $n.$
Since $ 0\st a\le a,$ we have 
\begin{align*}
\cgdnm{a\mI{A}}{n+1}&\le [ a\wedge (\mu(A)\st a)]\vee [\bar{y}\wedge (0\st a)]\le a\vee a=a
\end{align*}   
and the proof of \eqref{MO5} is complete.  From \eqref{MO5}, we obtain  $0\st \cgdnm{a\mI{A}}{n+1}\le 0\st a\le a\le \bar{y}$ for all $a\in Y.$
Hence,
\begin{align*}
\cgdnm{a\mI{A}}{n+1}&=\big[ a\wedge  (\mu(A)\st \cgdnm{a\mI{A}}{n})\big]\vee \big[0\st \cgdnm{a\mI{A}}{n}\big].
\end{align*}    
As $0\st \cgdnm{a\mI{A}}{n}\le 0\st a
\le a$ and $0\st\cgdnm{a\mI{A}}{n}\le \mu(A)\st \cgdnm{a\mI{A}}{n},$ we get 
\begin{align}\label{MO22}
\cgdnm{a\mI{A}}{n+1}&=a\wedge\big[\mu(A)\st \cgdnm{a\mI{A}}{n}\big].
\end{align}    
Applying induction on $n,$ we obtain the statement (c).
\end{proof}

To shorten the notation, we write  $\cgdn{f}{n}$ instead of $\cgdnm{f}{n}$ if there is no ambiguity. Hereafter, for a~link map $\st$ and $\mu\in\cMm$ we use the convention $\mu^\st _{k+1}(A):=\mu(A)\st \mu^\st_k(A)$ for all $k\ge 1$ with $\mu^\st _1(A)=\mu(A)$ provided that $\mu^\st_{k}(X)\in Y$ for all $k.$ It is evident that $\mu^\st_n$  is a~monotone measure if  $\mu^\st_n(X)>0.$  Some properties of the lower $n$-Sugeno integral are analogous to those of the upper $n$-Sugeno integral, which is shown in what follows.

\begin{pro} \label{pro4.3}
Let $(\mu,f)\in \cMm\times \cFf$ and $A\in \cA.$ 
\begin{enumerate}[noitemsep]
\item[(a)]  If $\rI{f}=0,$ then $\cgdn{f}{n}=0$ for all $n\ge 2.$ If $\cgdn{f}{k}=0$ for some $k>1$ and $a\st b>0$ for all $a,b>0,$  then $\cgdn{f}{n}=0$ for all $n\ge 1.$
\item[(b)] $\cgdn{f}{n}=\lim _{t\nearrow \cgdn{f}{n}}\big(t\wedge (\mu (\{f\ge t\})\st\cgdn{f}{n-1})\big)$ if  $\cgdn{f}{n-1}>0.$
\item[(c)] $\cgdn{f}{n}=\lim _{t\searrow \cgdn{f}{n}}\big(t\vee (\mu (\{f>t\})\st\cgdn{f}{n-1})\big)$ if $\cgdn{f}{n}<\bar{y}.$
\item[(d)]  There is $a>1$ such that $\cgdn[\nu]{ag}{n}\le a\cgdn[\nu]{g}{n}$ for all $(\nu,ag)\in \cMm\times \cFf$ and $n\ge 1,$ if $x\st (ay)\le a(x\st y)$ for all $x\in \nu(\cA)$ and $ay\in Y.$
Moreover, $\cgdn[\nu]{ag}{n}\ge a\cgdn[\nu]{g}{n}$ for some $a\in (0,1)$ and for all $(\nu,g)\in \cMm\times \cFf$ and $n\ge 1,$ if 
$x\st (ay)\ge a(x\st y)$  for all $x\in \nu(\cA)$ and $y\in Y.$
\item[(e)] $\cgdn{a\mI{A}}{n}=\mu^\st_n(A)$ for 
$a\in [\max_{1\le k\le n} \mu^{\st}_k(A),\bar{y}].$
\item[(f)] (Idempotency) 
$\cgdn{a\mI{X}}{n}=a$ for all $a\in Y$ and $n\ge 1$ if and only if $\mu(X)\ge \bar{y}$ and $\mu(X)\st a\ge a$ for all $a\in Y.$
\end{enumerate}
\end{pro}
\begin{proof} 
$(a)$ The first assertion follows from  Lemma \ref{lemat}\,(e) and the induction, since $0\st 0=0.$ If 
$\cgdn{f}{k}=0$ for some $k>1,$ then  
repeating similar argument as in the proof of Proposition 
\ref{pro3.3}\,(a), we get the assertion.

\noindent $(b)$ and $(c)$ By monotonicity of $\st,$ we have $t>\mu(\{f\ge t\})\st\cgdn{f}{n-1}$ for $t>\cgdn{f}{n}$ and $t<\mu(\{f\ge t\})\st\cgdn{f}{n-1}$ for $t<\cgdn{f}{n}.$ In consequence, both properties (b) and (c) hold. See the proof of Lemma \ref{lemat}\,(c)-(d). 

\noindent $(d)$ The proof is similar to that of Proposition~\ref{pro3.3}\,(b). 

\noindent $(e)$ and $(f)$  The proofs go by induction on $n$; see \eqref{MO22}.
\end{proof}

\begin{pro} \label{pro4.4}
The sequence $(\cgdn{f}{n})_{n=1}^\infty$ is nondecreasing for all $(\mu,f)\in \cMm\times \cFf$ if and only if $a\st b\ge a\wedge b$ for all $a\in [0,\infty]$ and $b\in Y.$
\end{pro}
\begin{proof}
``$\Rightarrow$'' 
By \eqref{MO22} for $n=1,$ we have  $\cgd{a\mI{A}}=
a\wedge [b\st (a\wedge b)],$
where $A\in \cA,$ $a\in Y$ and  $b=\mu(A).$ Since $\cgdn{a\mI{A}}{1}\le \cgdn{a\mI{A}}{2},$ we obtain 
$a\wedge b\le a\wedge [b\st (a\wedge b)]\le b\st a$
for all $a\in Y$ and $b\in [0,\infty],$ as desired.

\noindent ``$\Leftarrow$''  The proof is similar to that of Proposition \ref{pro3.4}, so we omit it.
\end{proof}

\begin{tw}\label{tw4.5}
For all $(\mu,f)\in \cMm\times \cFf$ and $n\ge 2$
\begin{align*}
\cgdn{f}{n}=\sup_{A\in\cA}\big\{\inf_{x\in A} f(x)\wedge \big(\mu(A)\st\cgdn{f}{n-1}\big)\big\}.
\end{align*}
\end{tw}
\begin{proof}
Use the same arguments as in the proof of Theorem~\ref{tw3.6}
for a~nondecreasing map $\st$
 in the first coordinate. 
\end{proof}


The following extension of the first equality in \eqref{sug_rown1} 
will be needed to prove the subadditivity property of the lower integral.

\begin{tw}\label{tw4.6} Assume that $\st$ is 
a~continuous link map in the first argument. 
Then for each $(\mu,f)\in \cMm \times \cFf$ and $n\ge 2$ we have
\begin{align*}
\cgdn{f}{n}=\inf_{t\in Y} \big\{t\vee \big(\mu (\{f> t\})\st \cgdn{f}{n-1}\big)\big\}.
\end{align*}
\end{tw}
\begin{proof}
\noindent The proof  is given in Appendix.
\end{proof}

Next, we show that the lower $n$-Sugeno integral is also a~subadditive functional under some extra assumptions. 

\begin{tw}\label{tw4.7} 
Suppose that  $\trdd\colon \mu(\cA)\times \mu(\cA)\to \mu(\cA),$ $f,g\in \cFf$ and
\begin{align}\label{MO1}
[(a+b)\wedge \bar{y}]\vee \big[(\alpha\trd \beta)\circledast ((c+d)\wedge \bar{y})\big]\le [a\vee (\alpha\circledast c)]+[b\vee (\beta\circledast d)]
\end{align}
for $a,b,c,d\in Y$ and $\alpha,\beta\in \mu(\cA)$  with $\circledast\in\{\st,\mathrm{P}\},$ where $\st$ is a~continuous link map in the first argument and $x\mathrm{P}y=x$ for any $x,y.$ If $f,g$ are $\mu$-$\trdd$-subadditive and  $f+g\in \cFf,$  then 
\begin{align}\label{sug_rown}
\cgdn{f+g}{n}\le\cgdn{f}{n}+\cgdn{g}{n}
\end{align}
for all $n\ge 1.$
Moreover, if
\eqref{sug_rown} is valid for all $f,g\in \cFf$ 
such that $f+g\in \cFf$ and $n$ such that $\mu^{\st}_n(X)>0,$
then the monotone measure $\mu^\st_n$ is subadditive.
\end{tw}
\begin{proof} 
We use the induction by $n.$ The proof of subadditivity of the Sugeno integral (the case $n=1$ and $\circledast=\mathrm{P}$) can be found in the proof of Theorem~\ref{tw3.8}.
Assume that  inequality \eqref{sug_rown} holds for some $n\ge 2$ and all $\mu$-$\trdd$-subadditive functions $f,g.$ Combining inductive hypothesis and \eqref{MO1} with $\circledast=\star$ yields 
\begin{align*}
[(a+b)\wedge \bar{y}]\vee &\big[\mu(\{f+g>(a+b)\wedge \bar{y}\})\st (\cgdn{f+g}{n}\wedge \bar{y})\big]\notag
\\&\le \big[(a+b)\wedge \bar{y}\big]\vee \big[\big(\mu (\{f> a\})\trd\mu (\{g> b\})\big)\st \big((\cgdn{f}{n}+\cgdn{g}{n})\wedge \bar{y}\big)\big]\nonumber
\\&\le \big[a\vee \big(\mu (\{f> a\})\st \cgdn{f}{n}\big)\big]
+\big[b\vee \big(\mu (\{g> b\})\st \cgdn{g}{n}\big)\big]
\end{align*}
for all $a,b\in Y.$ By Theorem~\ref{tw4.6}, we get 
\begin{align*}
\cgdn{f+g}{n+1}&\le \big[a\vee \big(\mu (\{f> a\})\st \cgdn{f}{n}\big)\big]
+\big[b\vee \big(\mu (\{g> b\})\st \cgdn{g}{n}\big)\big]
\end{align*} 
 for any $a,b\in Y.$
Taking infimum over $a$ and then with $b$ gives \eqref{sug_rown}.

Put $f=a\mI{A}$ and $g=a\mI{B\backslash A}$ in \eqref{sug_rown}, where $a\ge \max_{k\le n}\mu^\st_k(A\cup B).$ Then from Proposition~\ref{pro4.3}\,(e) and by monotonicity of $\mu^\st_n,$ we get $\mu^\st_n(A\cup B)\le \mu^\st_n(A)+\mu^\st_n(B\backslash A)\le \mu^\st_n(A)+\mu^\st_n(B),$ as desired. 
\end{proof}

\begin{example}\label{ex:4.8}
There are many link maps $\st$  with $Y=[0,\infty]$ such that subadditivity of $\mu_n^\st$  with arbitrary $n$ implies subadditivity of $\mu.$ For instance,  all idempotent operators (e.g. $a\st b=a^pb^{1-p}$ 
and $a\st b=p(a\wedge b)+(1-p)(a\vee b)$ with $p\in (0,1)$) as well as mappings
\begin{tasks}(4)
\task $a\st b=f(a)f(b),$  
\task $a\st b=g(ab),$
\task $a\st b=h(a+b),$
\task $a\st b=(a^{q}+b^{q})^{1/q},$ 
\end{tasks}
where $q>0$ and $f,g,h\colon [0,\infty]\to [0,\infty]$ are increasing superadditive functions\footnote{ A~function $f$ is superadditive if $f(a+b)\ge f(a)+f(b)$ for any $a,b.$} vanishing at $0,$  such that $h(x)\le x$ for all $x.$  In order to prove a) and b) one can use the inequality $a^n+b^n\le (a+b)^n.$
\end{example}

\begin{corollary}
If $\mu$ is a~subadditive monotone measure, or $f,g\in \mathcal{F}_{(X,[0,\infty])}$ are comonotone functions, then  
\begin{align}\label{low_subn}
\cgdn[+]{f+g}{n}\le\cgdn[+]{f}{n}+\cgdn[+]{g}{n}
\end{align}
for all $n\ge 1.$
Moreover, if 
\eqref{low_subn} is valid for some $n$ and all $f,g\in \mathcal{F}_{(X,[0,\infty])},$ then $\mu$ is subadditive.
\end{corollary}
\begin{proof} Put  $Y=[0,\infty],$ $\st =+$ and 
$a\trd b=(a+b)\wedge \mu(X)$
 or $\trdd=\vee$ in Theorem~\ref{tw4.7} and use Example~\ref{ex:4.8}\,(d) with $q=1.$
\end{proof}

Next, we give a~partial solution to the problem posed in \cite{mesiar13a}. 
The question is how to compute the {\it pseudo-decomposition integral of $n$-th order} defined as 
\begin{align}\label{def}
\bmm{n}{\oplus}{\odot}{f}=\sup\Big\{\bigoplus_{i=1}^n (a_i\odot\mu(A_i))\colon \bigoplus_{i=1}^n a_i\mI{A_i}\le f,\; a_i\in Y,\; 
A_1\subset \ldots \subset A_n \Big\}
\end{align}
based on a~pseudo-addition $\oplus\colon Y^2\to Y$ and a~$\oplus$-fitting pseudo-multiplication $\odot\colon Y\times [0,\mu(X)]\to Y$ (see \cite[Definition 3.1 and 3.4]{benvenuti}). 
The integral $\text{I}^{\oplus,\odot}_n$ is also called the \textit{Benvenuti integral of $n$-th order}. Our aim is to compute the integral 
$\bmm{n}{+}{\wedge}{f}.$
By the definition \eqref{def} we get 
\begin{align}\label{MU1}
\bmm{n}{+}{\wedge}{f}
&=\sup\Big\{\sum_{i=1}^n (a_i\wedge\mu(A_i))\colon\sum_{i=1}^n a_i\mI{A_i} \le f,\, A_{1}\subset \ldots \subset A_{n}\Big\}\nonumber
\\&=\sup\Big\{\sum_{i=1}^n (a_i\wedge\mu(A_i))\colon\sum_{i=1}^n\Big(\sum_{k=i}^n a_k\Big)\mI{A_i\setminus A_{i-1}} \le f,\,  A_{1}\subset \ldots \subset A_{n}\Big\}\nonumber
\\&=\sup\Big\{\sum_{i=1}^n \big((b_i-b_{i+1})\wedge \mu(\{ f\ge b_i\})\big)\colon  0=b_{n+1}\le b_n\le \ldots\le b_1\le \bar{y}\Big\}
\end{align}
with $b_i=\sum_{k=i}^n a_k$   and $A_0:=\emptyset,$
but computation of the integral from formula \eqref{MU1} is still a~difficult task. However, there is a~connection with the lower $n$-Sugeno integral.


\begin{tw}\label{tw4.10}
For all  $(\mu,f)\in \cMm \times\cFf$ and $n\ge 1$ 
\begin{align}\label{sub11}
\bmm{n}{+}{\wedge}{f}=\cgdn[+]{\mu,f}{n}.
\end{align}
\end{tw}
\begin{proof}
See Appendix.
\end{proof}

Combining the definition of $\cgdn[+]{\mu,f}{n}$ with Theorem \ref{tw4.10} gives  the following simple recurrence scheme 
\begin{align*} 
\bmm{n}{+}{\wedge}{f}=
\sup_{y\in Y}\big\{y\wedge \big(\mu(\{f\ge y\})+\bmm{n-1}{+}{\wedge}{f}\big)\big\}
\end{align*}
with $\bmm{1}{+}{\wedge}{f}:=\rI{f}.$

\begin{example}
Let $X=[0,1],$ $f(x)=x$ and $\mu (A)=(\lambda(A))^{1/2},$ where $\lambda$ is the Lebesgue measure. Then 
$$
\text{I}_n=\frac{2\,\text{I}_{n-1}-1+\sqrt{5-4\,\text{I}_{n-1}}}{2},\quad n=1,2,\ldots,
$$ 
where $\text{I}_n:=\bmm{n}{+}{\wedge}{f}$ and $\text{I}_0:=0.$ If $\mu (A)=(\lambda (A))^2,$ then
$$
\text{I}_n=\frac{3-\sqrt{5-4\,\text{I}_{n-1}}}{2},\quad n=1,2,\ldots.
$$
\end{example}

The next result provides a~connection between the lower $n$-Sugeno integral and the generalized Choquet integral introduced in~\cite{bustince} and deeply studied in~\cite{mesiar15}.

\begin{tw}\label{tw4.12}
Let $(\mu,f)\in\cMm\times\cFf$ and $n\ge 2.$ 
The lower $n$-Sugeno integral  can be represented as
\begin{align*}
\cgdn[+]{f}{n}&=
\inf\Big\{\sum_{i=1}^n \big((b_i-b_{i+1})\vee \mu(\{ f> b_i\})\big)\colon  0=b_{n+1}\le b_n\le \ldots\le b_1 \le \bar{y}\Big\}.
\end{align*}
\end{tw}
\begin{proof}
See Appendix.
\end{proof}

The lower $2$-Sugeno integral is neither maxitive  nor minitive homogeneous functional.

\begin{example}
Consider $f$ as in the Example~\ref{ex5}. Let $\mu(B)=0{.}25$ and $\mu(X)=1.$ Then
 $\cgd[+]{f}=0{.}5$
 and $\cgd[+]{(1/3)\mI{X}\vee f}=7/12,$ so $\cgd[+]{(1/3)\mI{X}\vee f}>(1/3)\vee \cgd[+]{f}.$
\end{example}


\begin{example}
Let $f=0{.}5\mI{A},$ $\mu(A)=0{.}5,$  $\mu(X)=1$ and $\st =\cdot.$ After simple calculations, we get $\cgd[\cdot]{f}=0{.}25$ and $\cgd[\cdot]{0{.}1\mI{X}\wedge f}=0{.}05.$
Thus, $\cgd[\cdot]{0{.}1\mI{X}\wedge f}<0{.}1 \wedge \cgd[\cdot]{f}.$ 
\end{example}

Now we give a~sufficient condition for  minitive/maxitive homogeneity.
\begin{pro}\label{pro4.15}
Suppose that $n\ge 2$ and  
$(\mu,f)\in \cMm\times \cFf.$
\begin{enumerate}[noitemsep]
\item[(a)]    
$\cgdn{a\mI{X}\wedge f}{n}=a\wedge \cgdn{f}{n}$  for any $a\in Y$
if 
$a\wedge (b \st (a\wedge c))=a\wedge (b\st c)$  for all $a,c\in Y$ and $b\in \mu(\cA).$
\item[(b)] 
$\cgdn{a\mI{X}\vee f}{n}=a\vee \cgdn{f}{n}$ for any $a\in Y$ if  $\mu(X)\ge \bar{y}$ 
and  
$b\wedge c \le b \st c \le b\vee c$  for all $b\in [0,\infty]$ and $c\in Y.$
 \end{enumerate}
\end{pro}
\begin{proof}  The proof of part $(a)$ is similar to that of Proposition~\ref{pro3.12}\,(a) (applying Theorem~\ref{tw4.5}), so we omit it.  \\
Now we show $(b)$ by induction. The proof for $n=2$ will be omitted as it is quite similar to  the proof of the second induction step. 
From Theorem~\ref{tw4.5} and the induction hypothesis, we have
\begin{align*}
\cgdn{a\mI{X}\vee f}{n}&=\sup_{A} \big\{ (a\vee \inf_{x\in A} f(x)) \wedge \big[\mu(A)\st (a\vee \cgdn{f}{n-1})\big]\big\}
\\&=\sup_{A} \big\{ (a\vee \inf_{x\in A} f(x)) \wedge \big[(\mu(A)\st a)\vee (\mu(A)\st \cgdn{f}{n-1})\big]\big\}
\\&=\sup_A\big\{ [a\wedge (\mu(A)\st a)]\vee \big[a\wedge (\mu(A)\st \cgdn{f}{n-1})\big]
\\&\qquad \vee [\inf_{x\in A} f(x) \wedge(\mu(A)\st a)]\vee \big[\inf_{x\in A} f(x)\wedge (\mu(A)\st \cgdn{f}{n-1})\big]\big\},
\end{align*}
where we write $\sup_A$ instead of $\sup_{A\in\cA}.$ Furthermore
\begin{align*}
\cgdn{a\mI{X}\vee f}{n}&=[a\wedge (\mu(X)\st a)]\vee  [a\wedge (\mu(X)\st \cgdn{f}{n-1})]\vee \sup_A\big\{ \inf_{x\in A} f(x) \wedge (\mu(A)\st a)\big\} \notag
\\&\qquad \vee  \sup_A\big\{\inf_{x\in A} f(x)\wedge (\mu(A)\st \cgdn{f}{n-1})\big\}.\notag
\end{align*}
By the assumption that $\mu(X) \st a\ge \mu(X)\wedge a=
a$ and the fact that  $a\wedge(\mu(X)\st \cgdn{f}{n-1})\le a,$ we have
\begin{align}\label{MO6}
\cgdn{a\mI{X}\vee f}{n}&=a\vee \sup_A\{ \inf_{x\in A} f(x) \wedge (\mu(A)\st a) \}\vee  \cgdn{f}{n}.
\end{align}
Observe that
\begin{align*}
  \sup_A\big\{ \inf_{x\in A} f(x) \wedge (\mu(A)\st a)\big\}&\le   \sup_A\big\{ \inf_{x\in A} f(x) \wedge (\mu(A)\vee a)\big\}
  \\&=\sup_A\big\{ (\inf_{x\in A} f(x) \wedge \mu(A))\vee (\inf_{x\in A} f(x)\wedge a)\big\}
  \\&\le \rI{f}\vee a\le \cgdn{f}{n}\vee a,
\end{align*}
where the last inequality follows from Proposition~\ref{pro4.4}. By~\eqref{MO6}, we obtain $
\cgdn{a\mI{X}\vee f}{n}=a\vee \cgdn{f}{n},$ as desired.
\end{proof}

The condition in $(a)$ is satisfied if $x\st a\ge a$ for any $x$  and $a,$ e.g. $x\st a:=(x^p+a^p)^{1/p}$ for $p>0.$ On the other hand, any OWA operator of the form $x\st y= p(x\wedge y)+(1-p)(x\vee y)$ satisfies the condition in $(b)$ for $p\in [0,1].$

\section{Applications}\label{sec:application}
\subsection{(A) Scientometric indices}

We put $X=\mN,$ where $\mN=\{1,2,\ldots\}$ denotes the set of all positive integers, and $\mu\colon 2^\mN \to [0,\infty]$ is the counting measure, i.e., $\mu(A)=\mathrm{Card}(A)$ for any $A\in 2^\mN.$ A~scholar with some publications is formally described by an infinite vector $\bx=(x_1,x_2,\ldots),$ called a~\textit{scientific record},
where $x_i\in\mN_0$ with $\mN_0=\mN\cup\{0\}$ such that $x_{1}\ge x_{2}\ge\ldots.$
The positive value of $x_i$ gives the number of citations of $i$-th scholar publication, and
the value $x_i=0$ means either a~paper with zero citations or a~nonexisting paper. 
From now on we consider the scientific records with $x_1\ge 1.$
The $h$-index of $\bx$ is defined as follows \cite{hirsch,mesiar14}
\begin{align*}
\rH(\bx) = \max\{k\colon x_k\ge k\} = \max_{k} \{k\wedge x_k\}.
\end{align*}
Note that there are several papers characterizing the Hirsch index via various axioms, e.g.~\cite{brandao,Miroiu,Quesada2,woeginger}. An interesting axiom $\rH(\bx)=\rH(\by_\bx)$ is called the \textit{symmetry} of the $h$-index, see \cite[Proposition 3.1]{woeginger2}. Here $\by_\bx:=(y_1,y_2,\ldots)$ is called the \textit{conjugate} of $\bx$ with $y_i=\sum_{k=1 }^\infty \mI{\{x_k\ge i\}}$ providing the number of publications with at least $i$ citations. 

As it is well known, the Hirsch index has some drawbacks. In order to compensate some defects of $h$-index, many authors have introduced new scientometric indices that lead to better discrimination of some types of data than $h$-index (see \cite{mesiar14,mingers,woeginger}).
Here we discuss a~few of them and show that the upper/lower $n$-Sugeno integrals 
generalize some known scientometric indices. Firstly, recall that Narukawa and Torra~\cite{torra2} have shown that $h$-index is the Sugeno integral with respect to counting measure. In consequence, the upper/lower $n$-Sugeno integral generalizes $h$-index too. 

\paragraph{(i) Generalized Kosmulski index}

There are several modifications of $h$-index based on the input $k,$ e.g.
$h_\lambda(\bx)=\max\{k\colon x_k\ge \lambda k\}$ of Van Eck~\cite{eck}, $h(2)$-index $\mathsf{H2}(\bx) = \max\{k\colon x_k\ge k^2\}$ of Kosmulski \cite{kosmulski}, or its extended version $\max\{k\colon x_k\ge k^m\}$ with $m=3,4,\ldots.$ In general, for any nondecreasing function $s\colon[0,\infty]\to [0,\infty]$ the \textit{generalized Kosmulski index} is given by $\mathsf{K}_s(\bx) = \max \{k\colon x_k\ge s(k)\},$ see \cite{deineko}\footnote{In order to get an integer-valued index, in the original paper authors consider the function $s\colon\mN_0\to \mN_0$ 
with $s(0)=0$ and $s(k)\ge 1$ for each $k\in\mN.$}.

Now we will show the connection between generalized Kosmulski index and upper/lower $2$-Sugeno integral. For this purpose consider
$\c_s\colon [0,\infty]^2\to [0,\infty]$ 
defined as $a\c_{s} b=s(a)$ with $s\colon [0,\infty]\to [0,\infty]$
being a~nondecreasing function such that $s(0)=0.$ Immediately, $\c_s$ is an admissible fusion as well as a~link function with $Y=[0,\infty].$ 
Note that each scientific record $\bx$ uniquely determines a~function $f\colon X\to \mN_0$ as $x_i=f(i),$ and vice versa. Hence, the notation $\cg[\circ_s]{\bx}$ is justified and
\begin{align*}
\cg[\circ_s]{\bx}=\max_k\{s(k)\wedge \mu(\{i\colon x_i\ge k\})\}=\max_j\{s(x_j)\wedge \mu(\{i\colon x_i\ge x_j\})\}.
\end{align*}
Since $\mu$ is the counting measure, we get
\begin{align}\label{chwilowa1}
\cgn[\circ_s]{\bx}{2}=\max_j\{s(x_j)\wedge j\}=\max_j\big\{s(\mu(\{i\colon y_i\ge j\}))\wedge j\big\}=\cgdn[\circ_s]{\by_\bx}{2},
\end{align}
where $\by_\bx$ is the conjugate of $\bx.$ Moreover,
\begin{align}\label{chwilowa2}
    \cgdn[\circ_s]{\bx}{2}&=\max_k\big\{k\wedge s(\mu(\{i\colon x_i\ge k\}))\big\}=\max_k\big\{x_k\wedge s(\mu(\{i\colon x_i\ge x_k\}))\big\} \notag   \\&=\max_k\{x_k\wedge s(k)\}=\max_k\{\mu(\{i\colon y_i\ge k\})\wedge s(k)\}=\cgn[\circ_s]{\by_\bx}{2}.
\end{align}
It is easy to see that $\cgdn[\circ_s]{\bx}{2}$ and $\cgn[\circ_s]{\bx}{2}$ for $s(a)=a$ (under the convention $s(\infty)=\infty$) coincide with the $h$-index of $\bx.$ In consequence, this proves Proposition~31 from \cite{woeginger2}, i.e., the symmetry $\rH(\bx)=\rH(\by_\bx)$ of $h$-index. However, the integrals $\cgn[\circ_s]{\cdot}{2}$ and  $\cgdn[\circ_s]{\cdot}{2}$ are not symmetric in general, i.e., the equalities
$\cgn[\circ_s]{\bx}{2}=\cgn[\circ_s]{\by_\bx}{2}$ and $\cgdn[\circ_s]{\bx}{2}=\cgdn[\circ_s]{\by_\bx}{2}$ need not hold for each $s$ and $\bx.$ 

\begin{example}
For $s(a)=2a$ and the scientific record $\bx=(3,0,\ldots)$ with $\by_\bx=(1,1,1,0,\ldots)$ we have $\cgn[\circ_s]{\bx}{2}=1=\cgdn[\circ_s]{\by_\bx}{2},$ but $\cgn[\circ_s]{\by_\bx}{2} = 2=\cgdn[\circ_s]{\bx}{2}.$ Note that for 
$s(a)=\lambda a$ with $\lambda>0$ it follows from \eqref{chwilowa2} that $\cgdn[\c_s]{\bx}{2}= \max_k\{x_k\wedge (\lambda k)\}.$
This index was introduced in \cite[Definition 2]{gag1}. 
\end{example}

\begin{pro} \label{pro:5.2}
Let $\c_s\colon [0,\infty]^2\to [0,\infty]$ be such that 
$a\c_{s} b=s(a)$ with $s\colon [0,\infty]\to [0,\infty]$ being an increasing and continuous function such that $s(0)=0.$
Then for each scientific record $\bx$ we have
\begin{itemize}[noitemsep]
    \item[(a)] $\mathsf{K}_s(\bx)=\cgn[\circ_{\lfloor\widehat{s}\rfloor}]{\bx}{2},$
    \item[(b)] $\mathsf{K}_s(\bx)=
\lfloor \cgn[\circ_{\widehat{s}}]{\bx}{2}\rfloor =\big\lfloor\widehat{s}\big(\cgdn[\circ_s]{\bx}{2}\big)\big\rfloor,$
\end{itemize}
where $\lfloor\cdot\rfloor$ is the floor function and $\widehat{s}= s^{-1}.$
\end{pro}
\proof
Observe that 
$\mathsf{K}_s(\bx)=\max\{k\colon s^{-1}(x_k)\ge k\}=\max\{k\colon \lfloor s^{-1}(x_k)\rfloor \ge k\}=\max_k\{k\wedge \lfloor s^{-1}(x_k)\rfloor\}.$
From \eqref{chwilowa1}, we get $\mathsf{K}_s(\bx)=\cgn[\circ_{\lfloor\widehat{s}\rfloor}]{\bx}{2},$ as $\c_{\lfloor \widehat{s}\rfloor}$ is an admissible fusion function. Next, note that
$\cgn[\circ_{\lfloor\widehat{s}\rfloor}]{\bx}{2}=\lfloor \max_k\{k\wedge s^{-1}(x_k)\}\rfloor,$  since $\lfloor k\wedge a\rfloor=k\wedge \lfloor a\rfloor$ for each $a\ge 0$ and $k\in \mN$ and $\max_k g(z_k)=g(\max_k z_k)$  for any nondecreasing function $g.$
Using \eqref{chwilowa1} again, we obtain $\mathsf{K}_s(\bx)=
\lfloor \cgn[\circ_{\widehat{s}}]{\bx}{2}\rfloor.$
Moreover,
$\mathsf{K}_s(\bx)=\lfloor s^{-1}(\max_k\{s(k)\wedge x_k)\})\rfloor = \big\lfloor\widehat{s}\big(\cgdn[\circ_s]{\bx}{2}\big)\big\rfloor,$ where the latter equality follows from \eqref{chwilowa2}.
\qed\medskip

All the above considerations are true also for the upper/lower $n$-Sugeno integral for any $n\ge 2.$


\paragraph{(ii) Upper and lower $2$-$h$-indices} 
We return back to the original indices our motivation comes from. Indeed, Mesiar and G\k{a}golewski \cite{mesiar14} introduced the upper $2$-$h$-index and the lower $2$-$h$-index of a~scientific record $\bx$ as follows:
 \begin{align}
 \rH^u_2(\bx)&=\max_k\big\{(k+\rH(\bx))\wedge x_{k}\big\}
,\label{hu}\\
 \rH^l_2(\bx)&=\rH(\bx)+\max_k\big\{(k-\rH(\bx))_+\wedge x_{k}\big\}=\max_k\big\{k\wedge (x_k+\rH(\bx))\big\},
\label{hd}
 \end{align}
where $a_+=\max(a,0).$
In other words, upper $2$-$h$-index is $h$-index increased by the value of $h$-index calculated for the scientist's output after removing $h$ citations from each work. On the other hand, lower $2$-$h$-index is $h$-index increased by the value of $h$-index of a~scientific record $\bx|_{\rH(\bx)}=(x_{\rH(\bx)+1}, x_{\rH(\bx)+2}, \ldots).$ The latter $h$-index of $\bx|_{\rH(\bx)}$ corresponds to $h$-index of $\bx$ without publications in the Hirsch core, cf.~\cite{Rousseau}. 

\begin{example}\label{apex:1}
Let $\bx=(6,6,4,3,1,1,1,0,\ldots).$ Clearly, $\rH(\bx)=3,$ and $\bx|_{\rH(\bx)}=(3,1,1,1,0,\ldots)$ is the scientific record obtained from $\bx$ after removing the Hirsch core, i.e., the first three papers. Since $\rH(\bx|_{\rH(\bx)})=1,$ we have $\rH_2^l(\bx)=4.$ Analogously, $\rH_2^u(\bx)=5,$ see Fig.~\ref{Fig1}.
\end{example}

\begin{figure}[h]
\centering
\includegraphics[scale=0.9]{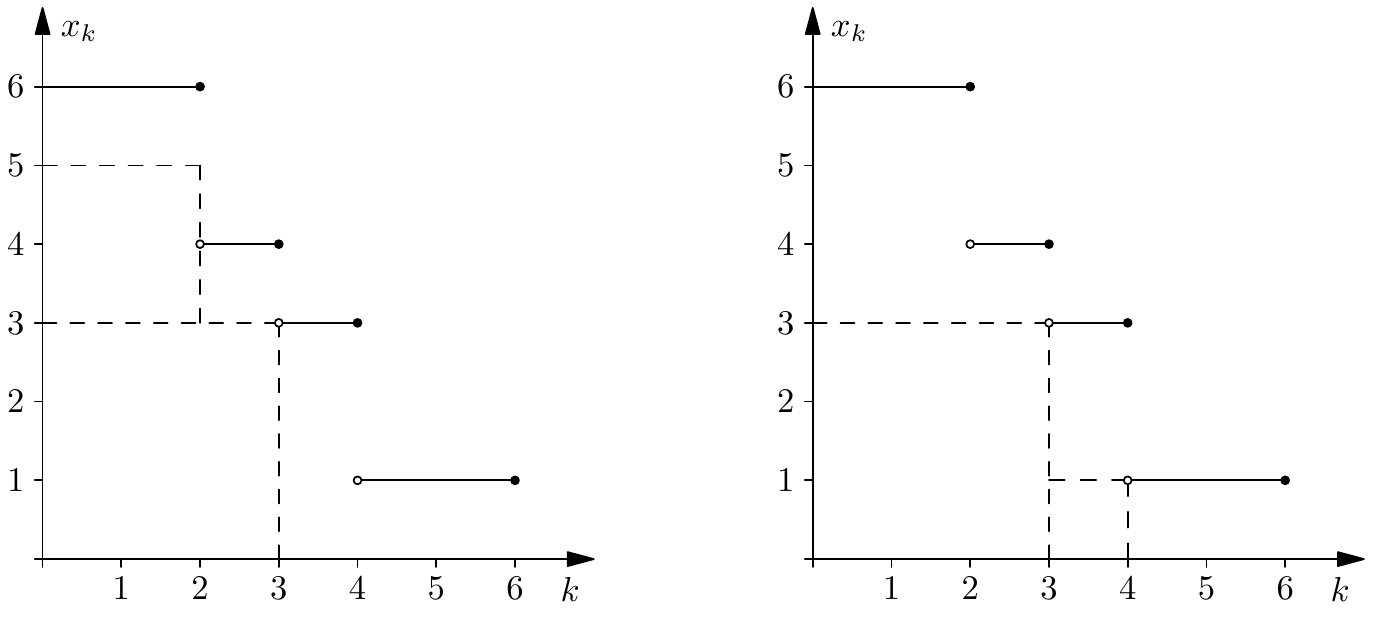}
\caption{
Illustration of formula \eqref{hu} for $\rH^u_2(\bx)=5$ (left) and \eqref{hd} for $\rH^l_2(\bx)=4$ (right).}
\label{Fig1}
\end{figure}

\begin{pro}
For each scientific record $\bx$ we have $\cg[+]{\bx} = \rH^l_2(\bx)=\cgdn[+]{\by_\bx}{2}$ and $\cgdn[+]{\bx}{2} = \rH^u_2(\bx)=\cg[+]{\by_\bx}.$
\end{pro}
\proof
For proving the statements, the admissible fusion function and the link function is $\c_s=+.$ Repeating the considerations from \eqref{chwilowa1} and \eqref{chwilowa2} we finish the proof.
\qed\medskip

\paragraph{(iii) $\rH_\alpha$ and $\rH^\beta$-indices}

Next we show that the indices $\rH_\alpha$ and $\rH^\beta$ recently introduced in \cite{JMS} as
\begin{align*}
 \rH_\alpha(\bx)&=\max_k \big\{\lfloor (x_k/\alpha)\wedge \mu(\{i\colon x_i\ge x_k\})\rfloor\big\}=\max_k \{\lfloor (x_k/\alpha)\wedge k\rfloor\},\qquad \alpha>0,\\
 \rH^\beta(\bx)& =\big\lceil \max_k \big\{x_k\wedge (\mu(\{i\colon x_i\ge x_k\})/\beta)\big\}\big\rceil=\lceil \max_k \{x_k\wedge (k/\beta) \}\rceil,\qquad \beta>0,
\end{align*}
are also a~special case of upper/lower Sugeno integral. Here, 
$\lceil\cdot\rceil$ is the  
ceiling of a~real number. 
Index $\rH_\alpha$ is able to compensate a~lower number of citations and $\rH^\beta$ compensates a~lower number of papers.

\begin{pro}\label{pro:5.6}
For each scientific record $\bx$ 
we have 
\begin{itemize}[noitemsep]
\item[(i)] $\rH_\alpha(\bx)= \cgn[\c_s]{\bx}{2}$ with $s(a)=\lfloor a/\alpha\rfloor,$
\item[(ii)] $\rH^\beta(\bx) =  \cgdn[\c_s]{\bx}{2}$
with $s(a)=\lceil a/\beta\rceil.$
\end{itemize}
\end{pro}
\proof 
(i) Based on \eqref{chwilowa1} we have $\cgn[\c_s]{\bx}{2}=\max_k\{\lfloor x_k/\alpha \rfloor \wedge k\}.$ To finish the proof one can use the fact that $\lfloor a\wedge k\rfloor=\lfloor a\rfloor \wedge k$ for each $a\ge 0$ and $k\in \mN.$

\noindent (ii) Using \eqref{chwilowa2} we get $\cgdn[\c_s]{\bx}{2}= \max_k \{x_k\wedge \lceil k/\beta\rceil\}.$ Since  $\lceil a\wedge k\rceil=\lceil a\rceil \wedge k$ for $a\ge 0$ and $k\in \mN,$ so $\cgdn[\c_s]{\bx}{2}= \max_k \{\lceil x_k\wedge  (k/\beta)\rceil\}.$ To get the statement, use 
$\max_k g(z_k)=g(\max_k z_k)$ for any nondecreasing  function $g.$
\qed\medskip

Using the similar arguments as in the proof of Proposition~\ref{pro:5.6}\,(ii) one can show that
$\rH_\alpha(\bx) = \lfloor \cgn[\c_s]{\bx}{2}\rfloor$
with $s(a)=a/\alpha$ and $\rH^\beta(\bx)=\lceil \cgdn[\c_s]{\bx}{2}\rceil$ with $s(a)=a/\beta.$ Thence and from Proposition~\ref{pro:5.2}\,(b) we conclude that $\rH_\alpha$ is a~special case of the generalized Kosmulski index $\mathsf{K}_s$ with $s(k)=\alpha k.$

\paragraph{(iv) Iterated $h$-index}
In 2009 in Garc\'{i}a-P\'{e}rez \cite{garcia,garcia1} considered a~multidimensional $h$-index 
and showed that the additional components are useful to distinguish individuals with the  same $h$-index. This approach has been studied further in~\cite{beal} in order to provide its axiomatic characterization. Formally, the iterated $h$-index $\riH$ of a~scientific record $\bx$ is a~vector $\riH(\bx)=(\riH_1(\bx),\riH_2(\bx),\ldots)$ with the components $\riH_n(\bx)$ defined for each $n\in\mathbb{N}$ by
$$\riH_n(\bx) = \max_k\big\{k\wedge x_{\riH_{0}(\bx)+\ldots+\riH_{n-1}(\bx)+k}\big\}$$
with $\riH_0(\bx):=0.$ Clearly, $\riH_{1}(\bx) = \rH(\bx)$ and $\riH_{1}(\bx)\ge \riH_{2}(\bx)\ge\ldots.$ Also, it is easy to see that $\riH_n(\bx)=\cgdn[+]{\bx}{n}-\cgdn[+]{\bx}{n-1}.$ Thus,

\begin{pro}\label{prop:iteratedH}
For each scientific record $\bx$ and each $n\in\mN$ we have $\cgdn[+]{\bx}{n} = \sum_{k=1}^n \riH_k(\bx).$
\end{pro}

\paragraph{(v) $p$-index and $c$-index}

It follows from Proposition~\ref{prop:iteratedH} that the functional defined by $\cgdn[+]{\bx}{\infty}:=\sup _n\cgdn[+]{\bx}{n}$ gives a~number of publications with at least one citation. This index is known as the \textit{$p$-index} (see~\cite[Definition 2.5]{eck}).
On the other hand, the number $\cgn[+]{\bx}{\infty}:=\sup _n\cgn[+]{\bx}{n}=x_1$ represents a~number of citations of the most important paper and it is called the \textit{$c$-index} \cite[Definition 2.6]{eck}, 
or the 
maximum-index \cite[Definition~2.5]{woeginger}. The $p$- and $c$-indices measure almost completely opposite aspects of the performance of a~researcher. The $p$-index can be seen as a~measure of productivity with focusing on productivity (i.e., number of papers) and paying almost no attention to impact (i.e., number of times a~paper has been cited). On the other hand, the $c$-index can be seen as a~measure of impact with focusing on impact and paying no attention at all to productivity. For instance, it prefers a~single highly cited paper over a~large number of slightly lower cited papers. Finally, the $s$-index defined by $\mathsf{s}(\bx) =\sum_{i=1}^\infty
x_i$ equals the total number of citations of all papers published by the scientist. Thus, the $s$-index takes into account all papers published by a~scientist and not only the most cited paper (as in $c$-index).

\subsection{(B) Aggregation functions}
Nowadays, aggregation processes naturally appear in almost every discipline and importance of aggregation functions may be seen in various applications including data fusion, decision making, computer science, social choice, etc. We shall show here that both the upper $n$-Sugeno and the lower $n$-Sugeno integrals are new aggregation functions under some mild assumptions on maps $\c$ and $\st$ with a~very natural max-min-type representations.
Firstly we recall the definition of an aggregation function.

\begin{definition}\cite[Definition~1.5]{beliakov}
A~function $\sA\colon [0,\bar{y}]^m\to [0,\bar{y}]$ is said to be an \textit{$m$-ary aggregation function}, if it is nondecreasing
and it satisfies the boundary conditions $\sA(0,\ldots,0)=0$ and $\sA(\bar{y},\ldots,\bar{y})=\bar{y}.$
\end{definition}

Put $X=\{1,2,\ldots,m\},$ $Y=[0,\bar{y}]$ 
and $\mu\in\cMm$ such that $\mu(X)\ge \bar{y}.$ 
For $n\ge 2$ the upper $n$-Sugeno integral $\cgn{\bx}{n}$ 
with $\bx=(x_1,\ldots,x_m),$ $x_i\in Y,$ and an admissible fusion map $\c$ satisfying $\bar{y}\c \bar{y}=\bar{y},$ is an $m$-ary aggregation function. Indeed, from Proposition~\ref{pro3.3}\,(d) we get 
$\cgn{\bar{y}\mI{X}}{n}=\bar{y}$
and $\cgn{0\mI{X}}{n}=0$ for all $n.$ Monotonicity follows from Proposition~\ref{pro3.2}\,(a). Additionally, the lower $n$-Sugeno integral $\cgdn{\bx}{n}$ for $n\ge 2$ is also an aggregation function 
if the link map $\st\colon [0,\infty] \times Y\to [0,\infty]$ is such that $\st\ge \wedge.$  In fact, the monotonicity follows from Proposition~\ref{pro4.2}\,(a), and by Proposition~\ref{pro4.3}\,(f) we have 
$\cgdn{\bar{y}\mI{X}}{n}=\bar{y}$
and $\cgdn{0\mI{X}}{n}=0.$
Moreover, for each $n\in\mN$ we obtain $$\cgn{\bx}{n+1} = \bigvee_{T\subset X}\big[\big((\bigwedge_{i\in T} x_i)\c \cgn{\bx}{n}\big)\wedge \mu(T)\big], \qquad \cgdn{\bx}{n+1} = \bigvee_{T\subset X}\big[(\bigwedge_{i\in T} x_i)\wedge \big(\mu(T)\st  \cgdn{\bx}{n}\big) \big]$$ (see Theorem~\ref{tw3.6} and Theorem \ref{tw4.5} for $\cA=2^X$)
providing the weighted max-min-type representations of the two sequences of aggregation functions.  

\medskip 
\section*{Conclusions}
Generalizing the upper and lower $2$-$h$-indices of Mesiar and G\k{a}golewski \cite{mesiar14} we have constructed upper and lower $n$-Sugeno integrals via iterating the Sugeno integral. These two classes of new functionals also include the generalized Kosmulski index \cite{deineko} and $\rH_\alpha$-index \cite{JMS}.
We have examined some of their universal mathematical properties that are useful in various fields such as scientometry, theory of integral and aggregation functions. Since there is only a~few number of papers combining the above fields, the present paper stimulates a~deeper study of the relationship between nonlinear functionals and scientometric indices. Thus, an applied research is supported by a~theoretical research.
 
As a~by-product, we have partially solved the question on computation of certain pseudo-decomposition integral providing the representation of Benvenuti integral of $n$-th order with respect to operations $\oplus=+$ and $\odot = \wedge$ as the lower $n$-Sugeno integral with respect to $+.$ So, our approach provides a~new way to look at pseudo-decomposition integrals and possibilities of their computation. 

\section*{Acknowledgement}
Authors would like to express their sincere thanks and gratitude to 
anonymous reviewers for their thoughtful suggestions toward the improvement of the paper.
This work was supported by the Slovak Research and Development Agency
under the  contract  No.  APVV-16-0337. The work is also cofinanced by bilateral call Slovak-Poland grant scheme No. SK-PL-18-0032 together with
the Polish National Agency for Academic Exchange under the contract No. PPN/BIL/2018/1/00049/U/00001.

\section*{Appendix}

\begin{proof}[\textbf{Proof of Theorems~\ref{tw3.5} and \ref{tw4.6}}] The arguments are similar to those of \cite[Theorem 2]{boczek9}. Put 
\begin{align*}
S_{n+1}&=\sup_{t\in Y}\left\{(t\c S_n)\wedge \big(\mu(\{ f\ge t\})\st S_n\big)\right\},\quad
Z_{n+1}=\inf_{t\in Y}\left\{(t\c Z_n)\vee \big(\mu(\{f>t\})\st Z_n\big)\right\}
\end{align*}
for all $n\ge 1$ with $S_1=\rI{f}=Z_1,$ 
where  
\begin{enumerate}[noitemsep]
\item[(A)]the map $\c$ is given in  Theorem~\ref{tw3.5} and $a\st b=a,$ or
\item[(B)] the map
$\st $ is defined in  Theorem~\ref{tw4.6} and $a\c b=a.$
\end{enumerate}
Assume that $S_1>0$ as if $S_1=0,$ then $S_n=0=Z_n$ for all $n$ (see  Propositions~\ref{pro3.3}\,(a) and \ref{pro4.3}\,(a)).
By induction we show that $S_n=Z_n$ for all $n.$ Clearly $S_1=Z_1.$ Suppose that $S_k=Z_k$  
for all $k\le  n.$ 
Set 
\begin{align*}
I:=\{t\in Y\colon \mu(\{ f\ge t\})\st S_n\ge t\c S_n\},\quad
J:=\{t\in Y\colon \mu (\{ f> t\})\st S_n\ge t\c S_n\}.
\end{align*}
Clearly, $J\subset I$ and $0\in J$ in the case (B). In the case (A) we have $a\st b=a,$ so we find that 
\begin{align*}
0\c S_n&\le S_n=\inf_{t\in Y} \big\{(t\c S_{n-1})\vee \mu (\{f> t\})\big\}\notag
\\&\le (0\c S_{n-1})\vee \mu (\{f> 0\})\le S_{n-1}\vee \mu(\{f>0\})\notag
\\&\le\ldots\le S_1\vee \mu (\{f>0\})=\mu(\{f>0\}),
\end{align*}
as $0\c y\le y$ and $S_1=\rI{f}=\inf _{t\in Y}\left\{t\vee\mu (\{ f>t\})\right\}\le \mu (\{f>0\}).$ In consequence,  $0\in J$ in both cases.
Since $t\mapsto t\c S$ is  nondecreasing and $t\mapsto \mu (\{ f> t\})\st S_n$ is nonincreasing, we have  $I=[0,a]$ or $I=[0,a)$ and $J=[0,b]$ or $J=[0,b)$ with $b\le a.$ We need to show that $a=b.$
Suppose  that $b<a.$ Hence by the definition of $I$ and $J,$ we have 
\begin{align*}
\mu (\{f> t\})\st S_n<t\c S_n\le  \mu (\{f\ge  t\})\st S_n
\end{align*}
for any $t\in (b,a).$
Let $b<d<c<a.$ As $\{f\ge c\}\subset \{f>d\},$ we get 
\begin{align*}
\mu (\{f\ge c\})\st S_n&\le \mu (\{f> d\})\st S_n
<d\c S_n\le c\c S_n\le \mu(\{f\ge c\})\st S_n,
\end{align*}
a~contradiction.

For each interval $D,$ let $D^c = [0,\bar{y}]\setminus D.$ 
 By continuity of the maps $t\mapsto t\c s$ and $t\mapsto t\st s$  
we obtain   
 \begin{align*}
S_{n+1}&=\sup_{t\in Y} \big\{(t\c S_n)\wedge \big(\mu (\{f\ge t\})\st S_n\big)\big\}=\sup _{t\in I}\{t\c S_n\}\vee \sup _{t\in I^c}\big\{\mu (\{f\ge t\})\st S_n\big\}
\\&=  \begin{cases}
\bar{y}\c S_n & \text{if } I=[0,\bar{y}],\\
 (a\c S_n)\vee \big(\mu(\{f\ge a^+\})\st S_n\big) & \text{if } I=[0,a],\, a<\bar{y},\\
(a\c S_n)\vee \big(\mu(\{f\ge a\})\st S_n\big)   & \text{if } I=[0,a),\, a\le \bar{y}
 \end{cases}
 \end{align*}  
with the convention that $\sup_{\emptyset}=0$ and $\inf_{ \emptyset}= \infty.$ Observe that  \begin{itemize}[noitemsep]
\item   if $I=[0,a]$ for $a<\bar{y},$ then from the definition of $I,$ $\mu (\{f\ge a^+\})\st S_n\le a\c S_n,$ so $S_{n+1}=a\c S_n,$ 
\item if $I=[0,a)$ for $a\le\bar{y},$ we have  $\mu (\{f\ge a\})\st S_n<a\c S_n,$ so $S_{n+1}=a\c S_n.$ 
\end{itemize}  
This implies that  $S_{n+1}=a\c S_n.$
Further, we have
\begin{align*}
Z_{n+1}&=\inf_{t\in Y} \big\{(t\c S_n)\vee \big(\mu (\{f> t\})\st S_n\big)\big\}=\inf _{t\in J}\{\mu (\{f>t\})\st S_n\} \wedge \inf _{t\in J^c}\{t\c S_n\}
\\&=
\begin{cases}
0\st S_n & \text{if } J=[0,\bar{y}],\\
(a\c S_n)\wedge \big(\mu (\{f>a\})\st S_n\big) & \text{if } J=[0,a],\,a<\bar{y},\\
(a\c S_n)\wedge \big(\mu (\{f>a^-\})\st S_n\big) & \hbox{if } J=[0,a),\,a\le\bar{y}
\end{cases}
\\&=
\begin{cases}
0\st S_n & \text{if } J=[0,\bar{y}],\\
a\c S_n & \text{if } J=[0,a],\,a<\bar{y}\; \hbox{ or }\; J=[0,a), \,a\le\bar{y},
\end{cases}
\end{align*} 
as
\begin{itemize}[noitemsep]
\item if $J=[0,a]$ and $a<\bar{y},$ then  $\mu (\{f>a\})\st S_n\ge a\c S_n,$ so $Z_{n+1}=a\c S_n,$ 
\item if $J=[0,a)$ and $a\le \bar{y},$ then $\mu (\{f>a^-\} )\st S_n\ge a\c S_n,$ and so $Z_{n+1}=a\c S_n.$ 
\end{itemize}
Consequently, we need to show that 
$S_{n+1}=Z_{n+1}$ if $J=[0,\bar{y}]=I.$ 
Indeed, we have $S_{n+1}=\bar{y}\c S_n$ and $Z_{n+1}=0\st S_n.$ Moreover, 
$$0\st S_n=\mu (\{f>\bar{y}\})\st S_n\ge \bar{y}\c S_n.$$
In the case (A), we have $0=0\st S_n\ge  \bar{y}\c S_n\ge 0,$ so $S_{n+1}=Z_{n+1}.$ In the case (B), $S_{n+1}=\bar{y},$ $Z_{n+1}=0\st S_n$ and $0\st S_n\ge \bar{y}.$
As $S_n\ge 0\st S_n$ and $S_n\le \bar{y},$ we get   
$0\st S_n=\bar{y},$ and so $S_{n+1}=Z_{n+1}.$ The proof is complete.
\end{proof}

\begin{proof}[\textbf{Proof of Theorem~\ref{tw4.10}}]
We begin with the formula \eqref{MU1}.
 It is clear that 
$\bmm{1}{+}{\wedge}{f}=\rI{f}=\cgdn[+]{f}{1}.$  
From Lemma~\ref{lemat}\,(e) it follows that if $\rI{f}=0,$ then $\mu(\{f\ge t\})=0$ for all $t>0.$ 
Hence by Proposition~\ref{pro4.3}\,(a),
$\bmm{n}{+}{\wedge}{f}=0=\cgdn[+]{f}{n}$ and the assertion holds for any $n.$

From now on let us assume $\rI{f}>0.$ To get a~better understanding, we first consider the case $n=2,$ that is, we show that 
$\cgd[+]{f}=\sup_{b_2\in Y}M_2(b_2),
$
where
$$
M_2(b_2)=(b_2\wedge\mu(\{ f\ge b_2\}))+\sup_{b_1\ge b_2}\big\{(b_1-b_2)\wedge  \mu (\{ f\ge b_1\})\big\}.
$$
Clearly, $b_2\wedge \mu (\{f\ge b_2\})\le \rI{f},$ so 
$b_2\wedge \mu (\{f\ge b_2\})\le b_2\wedge  \rI{f}$ and
\begin{align*}
M_2(b_2)&\le (b_2\wedge \rI{f})+\sup_{b_1\ge b_2\wedge \rI{f}} \big\{ \big(b_1-(b_2\wedge \rI{f})\big) \wedge\mu (\{ f\ge b_1\})\big\}
\\&=\sup_{b_1\ge b_2\wedge \rI{f}} \big\{b_1 \wedge \big(\mu (\{ f\ge b_1\})+(b_2\wedge \rI{f})\big)\big\} 
\\&\le \sup_{b_1\ge b_2\wedge \rI{f}} \big\{b_1 \wedge \big(\mu (\{ f\ge b_1\})+ \rI{f}\big)\big\}
\\&=\rI{f}+\sup_{b_1\ge b_2\wedge \rI{f}} \big\{ (b_1- \rI{f}) \wedge\mu (\{ f\ge b_1\})\big\}.
\end{align*}
Since $\sup_{b_1\in [b_2\wedge\rI{f},\,\rI{f}]} \big\{ (b_1- \rI{f}) \wedge\mu (\{ f\ge b_1\})\big\}=0,$ we have
\begin{align*}
M_2(b_2)&\le \sup_{b_1\ge \rI{f}} \big\{b_1 \wedge (\mu (\{ f\ge b_1\})+\rI{f})\big\}=\cgd[+]{f}.
\end{align*}
From the above it follows that
\begin{align}\label{l3}
\sup_{b_2\in Y} M_2(b_2)&\le \cgd[+]{f}.
\end{align}
Now we show that  the reverse inequality holds in \eqref{l3}. Recall that $\rI{f}>0.$ Evidently
$$\sup_{b_2\in Y}M_2(b_2)\ge \lim _{b_2\nearrow \rI{f}}M_2(b_2).$$ 
By Lemma  \ref{lemat}\,(c), we get   
\begin{align*}
\lim _{b_2\nearrow \rI{f}}M_2(b_2)&= \rI{f}+\lim _{b_2\nearrow \rI{f}}\sup_{b_1\ge b_2} \big\{(b_1-b_2)\wedge\mu (\{ f\ge b_1\})\big\}\nonumber
\\&\ge \rI{f}+\lim _{b_2\nearrow \rI{f}}\sup_{b_1\ge b_2}\big\{ (b_1-\rI{f})\wedge\mu(\{f\ge b_1\})\big\}\nonumber
\\&=\rI{f}+\sup_{b_1\ge \rI{f}}\big\{ (b_1-\rI{f})\wedge\mu(\{ f\ge b_1\})\big\}.
 \end{align*}
Thus, $\sup_{a_2\in Y}M_2(a_2)\ge \cgd[+]{f},$ so there is the equality in \eqref{l3}, as claimed.\\

Now, we show that the assertion \eqref{sub11} holds for all $n>2.$ Observe that  
\begin{align*}
\bmm{n}{+}{\wedge}{f}
=\sup_{b_n\in Y}  M_n(b_n,b_{n+1}),
\end{align*}
where $b_{n+1}=0$ and
$M_{n}$
is defined recursively using the formula
\begin{align*}
M_k(b_k,b_{k+1}):=\big[(b_k-b_{k+1})\wedge\mu (\{ f\ge b_k\})\big]+\sup_{b_{k-1}\ge b_k}  M_{k-1}(b_{k-1},b_k)
\end{align*}
for $k=2,\ldots,n$
with the initial condition $M_1(b_1,b_2)=(b_1-b_2)\wedge   
\mu (\{ f\ge b_1\}).$
Put $M^*_k(b_k):=\sup _{b_{k-1}\ge b_k}M_{k-1}(b_{k-1},b_k)$ for $k\ge 2.$ Mimicking the proof for $n=2,$ we obtain 
\begin{align*}
M_n(b_n,b_{n+1})&\le (b_n \wedge\rI{f})+\sup_{b_{n-1}\ge b_n }\big\{\big[(b_{n-1}-b_n)\wedge \mu(\{ f\ge b_{n-1}\})\big]+M_{n-1}^*(b_{n-1})\big\}
\\&\le (b_n \wedge\rI{f})
\\&\quad+\sup_{b_{n-1}\ge b_n \wedge\rI{f}}\big\{\big[\big(b_{n-1}-(b_n \wedge\rI{f})\big)\wedge \mu(\{ f\ge b_{n-1}\})\big]+M_{n-1}^*(b_{n-1})\big\}
\\&=\sup_{b_{n-1}\ge b_n \wedge\rI{f}}\big\{\big[b_{n-1}\wedge \big(\mu(\{ f\ge b_{n-1}\})+(b_n \wedge\rI{f})\big)\big]+M_{n-1}^*(b_{n-1})\big\}
\\&\le\sup_{b_{n-1}\ge b_n \wedge\rI{f}}\big\{\big[b_{n-1}\wedge \big(\mu(\{ f\ge b_{n-1}\})+\rI{f}\big)\big]+M_{n-1}^*(b_{n-1})\big\}
\\&=\sup_{b_{n-1}\ge b_n \wedge\rI{f}} N_{n-1}(b_{n-1}).
\end{align*}
Here and subsequently, 
\begin{align*}
N_k(b):=\big[b\wedge \big(\mu(\{ f\ge b\})+\cgdn[+]{f}{n-k}\big)\big]+M_k^*(b)
\end{align*}
for all $k=1,\ldots,n-1$
with the convention $M^*_1(b):=0.$
By the very definition of $\cgd[+]{f},$ we have  $b\wedge (\mu(\{ f\ge b\})+\rI{f})\le b\wedge \cgd[+]{f}$ for $b\in Y.$ Thus
\begin{align*}
N_{n-1}(b_{n-1}) &\le (b_{n-1}\wedge \cgd[+]{f})+\sup_{b_{n-2}\ge b_{n-1}} \big\{\big[(b_{n-2}-b_{n-1})\wedge \mu(\{f\ge b_{n-2}\}) \big]+M_{n-2}^*(b_{n-2})\big\}
\\&\le (b_{n-1}\wedge \cgd[+]{f})\\&\quad+\sup_{b_{n-2}\ge b_{n-1}\wedge \cgd[+]{f}} \big\{\big[\big(b_{n-2}-(b_{n-1}\wedge \cgd[+]{f})\big)\wedge \mu(\{f\ge b_{n-2}\}) \big]+M_{n-2}^*(b_{n-2})\big\}
\\&=\sup_{b_{n-2}\ge b_{n-1} \wedge\cgd[+]{f}}\big\{\big[b_{n-2}\wedge \big(\mu(\{ f\ge b_{n-2}\})+(b_{n-1} \wedge\cgdn[+]{f}{2})\big)\big]+M_{n-2}^*(b_{n-2})\big\}
\\&\le\sup_{b_{n-2}\ge b_{n-1} \wedge\cgd[+]{f}}\big\{\big[b_{n-2}\wedge \big(\mu(\{ f\ge b_{n-2}\})+\cgd[+]{f}\big)\big]+M_{n-2}^*(b_{n-2})\big\}
\\&=\sup_{b_{n-2}\ge b_{n-1} \wedge\cgd[+]{f}} N_{n-2}(b_{n-2})
\end{align*}
for any $b_{n-1}\ge b_n\wedge \rI{f}.$ In the same manner we obtain for $k=1,\ldots,n-2,$
\begin{align*}
\sup _{b_{n-k}\ge b_{n-k+1}\wedge\cgdn[+]{f}{k}}N_{n-k}(b_{n-k})\le \sup _{b_{n-k-1}\ge b_{n-k}\wedge\cgdn[+]{f}{k+1}}N_{n-k-1}(b_{n-k-1}).
\end{align*}
As a~consequence, we  get  
$$M_n(b_n,b_{n+1})\le \sup_{b_{n-1}\ge b_{n}\wedge \rI{f}}N_{n-1}(b_{n-1})\le\ldots\le\sup_{b_1\ge b_2\wedge\cgdn[+]{f}{n-1}} N_1(b_1)=\cgdn[+]{f}{n}$$
for all $b_n\in Y.$ Therefore,
\begin{align*}
\bmm{n}{+}{\wedge}{f}
\le \cgdn[+]{f}{n}.
 \end{align*} 
To finish the proof it is sufficient to show that $\sup_{b_n \in Y} M_n(b_n,b_{n+1})\ge \cgdn[+]{f}{n}.$ 
Using Lemma \ref{lemat}\,(c) and mimicking the proof for $n=2,$ we obtain
\begin{align}
\sup_{b_n\in Y} M_n(b_n,b_{n+1})&\ge \lim_{b_n\nearrow \rI{f}} M_n(b_n,b_{n+1})\notag
\\&=\rI{f}+\lim_{b_n\nearrow \rI{f}}\sup_{b_{n-1}\ge b_n}\big\{\big[(b_{n-1}-b_n)\wedge \mu(\{ f\ge b_{n-1}\})\big]+M^*_{n-1}(b_{n-1})\big\}\notag
\\&\ge  \rI{f}+\lim_{b_n\nearrow \rI{f}}\sup_{b_{n-1}\ge b_n}\big\{\big[(b_{n-1}-\rI{f})_{+}\wedge \mu(\{ f\ge b_{n-1}\})\big]+M^*_{n-1}(b_{n-1})\big\}\nonumber
\\&=\rI{f}+\sup_{b_{n-1}\ge \rI{f}}\big\{\big[(b_{n-1}-\rI{f})\wedge \mu(\{ f\ge b_{n-1}\})\big]+M^*_{n-1}(b_{n-1})\big\}\nonumber
\\&=\sup_{b_{n-1}\ge \rI{f}} N_{n-1}(b_{n-1}).\label{l4}
\end{align}
By Proposition~\ref{pro4.4},  $\cgd[+]{f}\ge \rI{f}>0.$   
Proposition~\ref{pro4.3}\,(b)  and \eqref{l4} implies 
\begin{align}
\sup_{b_n\in Y} M_n(b_n,b_{n+1})&
\ge\lim_{b_{n-1}\nearrow \cgd[+]{f}}\Big(\big[b_{n-1}\wedge \big(\mu(\{ f\ge b_{n-1}\})+\rI{f}\big)\big]+M_{n-1}^*(b_{n-1})\Big)\notag
\\&= \cgd[+]{f}+\lim_{b_{n-1}\nearrow \cgd[+]{f}}\sup_{b_{n-2}\ge b_{n-1}}M_{n-2}(b_{n-2},b_{n-1}).\label{l5a}
\end{align}
Next,  we get
\begin{align}\label{l5b}
\lim_{b_{n-1}\nearrow \cgd[+]{f}}&\sup_{b_{n-2}\ge b_{n-1}}M_{n-2}(b_{n-2},b_{n-1})\notag\\
&\ge \lim_{b_{n-1}\nearrow \cgd[+]{f}}\sup_{b_{n-2}\ge b_{n-1}}\big\{\big[(b_{n-2}-\cgd[+]{f})_+\wedge\mu(\{f\ge b_{n-2}\})\big]+M_{n-2}^*(b_{n-2})\big\}\notag\\
&\ge \sup_{b_{n-2}\ge \cgd[+]{f}}\big\{\big[(b_{n-2}-\cgd[+]{f})\wedge\mu(\{f\ge b_{n-2}\})\big]+M_{n-2}^*(b_{n-2})\big\}.
\end{align}
Thus from \eqref{l5a} and \eqref{l5b} we obtain
\begin{align*}
\sup_{b_n\in Y} M_n(b_n,b_{n+1})&\ge
\sup_{b_{n-2}\ge \cgd[+]{f}}N_{n-2}(b_{n-2}).
\end{align*}
Repeating the same reasoning we get 
\begin{align*}
\sup_{b_n\in Y} M_n(b_n,b_{n+1})\ge \sup_{b_{1}\ge \cgdn[+]{f}{n-1}}N_1(b_1)=\cgdn[+]{f}{n},
\end{align*}
as required.
The proof is complete.
\end{proof}

\begin{proof}[\textbf{Proof of Theorem~\ref{tw4.12}}]
Let $n\ge 2.$ We need to  show that $\cgdn[+]{f}{n}=L_n,$ where  
\begin{align*}
L_n&=\inf\Big\{\sum_{i=1}^n \big((b_i-b_{i+1})\vee \mu(\{ f> b_i\})\big)\colon  0=b_{n+1}\le b_n\le \ldots\le b_1 \le \bar{y}\Big\}.
\end{align*}
Clearly, $L_n=\inf_{b_n\in Y}  M_n(b_n,b_{n+1}),$
where 
$M_{n}$
is defined recursively using the formula
\begin{align*}
M_k(b_k,b_{k+1}):=\big[(b_k-b_{k+1})\vee\mu (\{ f>b_k\})\big]+\inf_{b_{k-1}\ge b_k}  M_{k-1}(b_{k-1},b_k)
\end{align*}
for any $k=2,\ldots,n$
with $M_1(b_1,b_2):=(b_1-b_2)\vee \mu (\{ f>b_1\}).$ 
For simplicity,  
put $M^*_k(b_k):=\inf _{b_{k-1}\ge b_k}M_{k-1}(b_{k-1},b_k)$ for $k=2,\ldots,n.$ Since   $b\vee \mu (\{f>b\})\ge \rI{f},$ we have 
$b\vee \mu (\{f>b\})\ge b\vee \rI{f}$ for all $b\in Y.$ Thus, 
\begin{align}\label{NO1}
M_n(b_n,b_{n+1})&\ge (b_n \vee\rI{f})
+\inf_{b_{n-1}\ge b_n \wedge\rI{f}}\left\{\big[\big(b_{n-1}-(b_n \vee \rI{f})\big)\vee \mu(\{ f>b_{n-1}\})\big]+M_{n-1}^*(b_{n-1})\right\}\notag\\&=\inf_{b_{n-1}\ge b_n \wedge\rI{f}}\left\{\big[b_{n-1}\vee \big(\mu(\{ f>b_{n-1}\})+(b_n \vee\rI{f})\big)\big]+M_{n-1}^*(b_{n-1})\right\}\notag\\
&\ge\inf_{b_{n-1}\ge b_n \wedge\rI{f}}\left\{\big[b_{n-1}\vee\big(\mu(\{ f>b_{n-1}\})+\rI{f}\big)\big]+M_{n-1}^*(b_{n-1})\right\}\notag\\ 
&=\inf_{b_{n-1}\ge b_n\wedge \rI{f}} N_{n-1}(b_{n-1}),
\end{align}
where 
$N_k(b):=\big[b\vee \big(\mu(\{ f>b\})+\cgdn[+]{f}{n-k}\big)\big]+M_k^*(b)$ for $k=1,\ldots,n-1$
with $M^*_1(b_1):=0.$
Theorem~\ref{tw4.6} gives that $b\vee (\mu(\{ f>b\})+\rI{f})\ge b\vee \cgd[+]{f}$ for all $b\in Y.$ Thus, 
\begin{align*}
N_{n-1}(b_{n-1}) &\ge (b_{n-1}\vee \cgd[+]{f})+\inf_{b_{n-2}\ge b_{n-1}}\big\{\big[(b_{n-2}-b_{n-1})\vee \mu(\{f> b_{n-2}\}) \big]+M_{n-2}^*(b_{n-2})\big\}
\\&\ge (b_{n-1}\vee \cgd[+]{f})
\\&\quad+\inf_{b_{n-2}\ge b_{n-1}\wedge \cgd[+]{f}} \big\{\big[\big(b_{n-2}-(b_{n-1}\vee \cgd[+]{f})\big)\vee\mu(\{f> b_{n-2}\}) \big]+M_{n-2}^*(b_{n-2})\big\}
\\&=\inf_{b_{n-2}\ge b_{n-1} \wedge\cgd[+]{f}}\big\{\big[b_{n-2}\vee \big(\mu(\{ f> b_{n-2}\})+(b_{n-1} \vee\cgdn[+]{f}{2})\big)\big]+M_{n-2}^*(b_{n-2})\big\}
\\&\ge\inf_{b_{n-2}\ge b_{n-1} \wedge\cgd[+]{f}}\big\{\big[b_{n-2}\vee\big(\mu(\{ f> b_{n-2}\})+\cgd[+]{f}\big)\big]+M_{n-2}^*(b_{n-2})\big\}
\\&=\inf_{b_{n-2}\ge b_{n-1} \wedge\cgd[+]{f}} N_{n-2}(b_{n-2})
\end{align*}
for any $b_{n-1}\ge b_n\wedge \rI{f}.$ Analogously,  for $k=1,\ldots,n-2$ we get 
\begin{align}\label{lali}
\inf_{b_{n-k}\ge b_{n-k+1}\wedge\cgdn[+]{f}{k}}N_{n-k}(b_{n-k})\ge \inf _{b_{n-k-1}\ge 
b_{n-k}\wedge\cgdn[+]{f}{k+1}}N_{n-k-1}(b_{n-k-1}),
\end{align}
where we use the fact that   for all $b\in Y$ and $k=2,\ldots,n-1$
\begin{align*}
 b\vee \big(\mu(\{f>b\})+\cgdn[+]{f}{k-1}\big) \ge b\vee \cgdn[+]{f}{k}
\end{align*}
(see Theorem~\ref{tw4.6}). As a~consequence of \eqref{lali} and \eqref{NO1}, we obtain  
\begin{align*}M_n(b_n,b_{n+1})&\ge \inf_{b_{n-1}\ge b_{n}\wedge \rI{f}}N_{n-1}(b_{n-1})\ge\ldots\ge\inf_{b_1\ge b_2\wedge\cgdn[+]{f}{n-1}} N_1(b_1)
\\&\ge \inf_{b_1\in Y} N_1(b_1)=\cgdn[+]{f}{n}
\end{align*}
 for all $b_n\in Y.$ Therefore $
L_n\ge \cgdn[+]{f}{n}.$ 
We show that $L_n\le \cgdn[+]{f}{n}.$ 

\noindent Let $\cgdn[+]{f}{n}=\bar{y}.$ Then 
$$
L_n\le M_n(\bar{y},b_{n+1})=\bar{y}+M^*_n(\bar{y})=\bar{y}=\cgdn[+]{f}{n}.
$$
Assume that $\cgdn[+]{f}{n}<\bar{y}.$  Using Lemma \ref{lemat}\,(d), we have
\begin{align}
L_n&\le \lim_{b_n\searrow \rI{f}} M_n(b_n,b_{n+1})\notag
\\&=\rI{f}+
\lim_{b_n\searrow \rI{f}}\inf_{b_{n-1}\ge b_n}\big\{\big[(b_{n-1}-b_n)\vee \mu(\{ f> b_{n-1}\})\big]+M^*_{n-1}(b_{n-1})\big\}\notag
\\&\le  \rI{f}+\lim_{b_n\searrow \rI{f}}\inf_{b_{n-1}\ge b_n}\left\{\big[(b_{n-1}-\rI{f})\vee \mu(\{ f>b_{n-1}\})\big]+M^*_{n-1}(b_{n-1})\right\}\nonumber
\\&=\rI{f}+\inf_{b_{n-1}\ge \rI{f}}\left\{\big[(b_{n-1}-\rI{f})\vee \mu(\{ f> b_{n-1}\})\big]+M^*_{n-1}(b_{n-1})\right\}\nonumber
\\&=\inf_{b_{n-1}\ge \rI{f}} N_{n-1}(b_{n-1}).\label{l4a}
\end{align}
From~\eqref{l4a}, Propositions~\ref{pro4.4} and~\ref{pro4.3}\,(c) we get \begin{align}
L_n&\le\inf_{b_{n-1}\ge\rI{f}} N_{n-1}(b_{n-1})\le\inf_{b_{n-1}\ge\cgd[+]{f}} N_{n-1}(b_{n-1})\notag
\\&\le\lim_{b_{n-1}\searrow \cgd[+]{f}}\big(\big[b_{n-1}\vee \big(\mu(\{ f> b_{n-1}\})+\rI{f}\big)\big]+M_{n-1}^*(b_{n-1})\big)\notag
\\&= \cgd[+]{f}+\lim_{b_{n-1}\searrow \cgd[+]{f}}M_{n-1}^*(b_{n-1}).\label{l5c}
\end{align}
Next, we have
\begin{align}\label{l5d}
\lim_{b_{n-1}\searrow \cgd[+]{f}} M_{n-1}^*(b_{n-1})&\le \lim_{b_{n-1}\searrow \cgd[+]{f}}\inf_{b_{n-2}\ge b_{n-1}}\big\{\big[(b_{n-2}-\cgd[+]{f})\vee\mu(\{f>b_{n-2}\})\big]+M_{n-2}^*(b_{n-2})\big\}\notag
\\ & \le \inf_{b_{n-2}\ge \cgd[+]{f}}\big\{\big[(b_{n-2}-\cgd[+]{f})\vee\mu(\{f>b_{n-2}\})\big]+M_{n-2}^*(b_{n-2})\big\}.
\end{align}
Thus, from \eqref{l5c} and \eqref{l5d} we obtain
\begin{align*}
L_n&\le\inf_{b_{n-1}
\ge \cgd[+]{f}}N_{n-2}(b_{n-2}).
\end{align*} 
Repeating the same reasoning we get
\begin{align*}
L_n&\le \inf_{b_{1}\ge \cgdn[+]{f}{n-1}}N_1(b_1)
=\cgdn[+]{f}{n-1}+\inf_{b\ge \cgdn[+]{f}{n-1}}\big\{(b-\cgdn[+]{f}{n-1})\vee \mu(\{ f>b\})\big\}=\cgdn[+]{f}{n},
\end{align*}
as required. 
\end{proof}

\end{document}